\documentclass[12pt]{amsart}
\usepackage{amsmath,amssymb,latexsym,esint,cite,mathrsfs}
\usepackage{verbatim,wasysym}
\usepackage[left=2.6cm,right=2.6cm,top=3cm,bottom=3cm]{geometry}
\makeatletter \@addtoreset{equation}{section} \makeatother

\usepackage{subfigure}
\usepackage{graphicx,epic,eepic}

\usepackage{color,enumitem,graphicx}
\usepackage[colorlinks=true,urlcolor=blue,
citecolor=red,linkcolor=blue,linktocpage,pdfpagelabels,
bookmarksnumbered,bookmarksopen]{hyperref}
\numberwithin{equation}{section}
\newtheorem{theorem}{Theorem}[section]
\newtheorem{lemma}[theorem]{Lemma}

\newtheorem{remark}[theorem]{Remark}
\numberwithin{equation}{section}

\allowdisplaybreaks

\begin{document}

\title[Kirchhoff type elliptic system]
%{On a class of Kirchhoff type elliptic system with exponential growth nonlinearities}
{Existence and multiplicity of solutions to a Kirchhoff type elliptic system with Trudinger-Moser growth}

\author[S. Deng]{Shengbing Deng}
\address{\noindent S. Deng-School of Mathematics and Statistics, Southwest University,
Chongqing 400715, People's Republic of China}\email{shbdeng@swu.edu.cn}

\author[X. Tian]{Xingliang Tian}
\address{\noindent X. Tian-School of Mathematics and Statistics, Southwest University,
Chongqing 400715, People's Republic of China.}\email{xltianswumaths@163.com}

\maketitle

\allowdisplaybreaks

\noindent {\bf Abstract}:
This paper deals with the existence and multiplicity of solutions for a class of Kirchhoff type elliptic systems involving nonlinearities with Trudinger-Moser exponential growth. We first study the existence of solutions for the following system:
\begin{eqnarray*}
    \left\{ \arraycolsep=1.5pt
       \begin{array}{ll}
        -\big(a_1+b_1\|u\|^{2(\theta_1-1)}\big)\Delta u= \lambda H_u(x,u,v)\ \ \ &\ \mbox{in}\ \ \ \Omega,\\[2mm]
        -\big(a_2+b_2\|v\|^{2(\theta_2-1)}\big)\Delta v= \lambda H_v(x,u,v)\ \ \ &\ \mbox{in}\ \ \ \Omega,\\[2mm]
        u=0, v=0\ \ \ \ &\ \mbox{on}\ \ \ \partial\Omega,
        \end{array}
    \right.
    \end{eqnarray*}
    where $\Omega$ is a bounded domain in $\mathbb{R}^2$ with smooth boundary,\ $\|w\|=\big(\int_{\Omega}|\nabla w|^2dx\big)^{1/2}$, $H_u(x,u,v)$ and $H_v(x,u,v)$ behave like $e^{\beta |(u,v)|^2}$ when $|(u,v)|\rightarrow \infty$ for some $\beta>0$, $a_1, a_2>0$, $b_1, b_2> 0$, $\theta_1, \theta_2> 1$ and $\lambda$ is a positive parameter.
    In the later part of the paper, we also discuss a  new multiplicity result for the above system with a positive parameter induced by the nonlocal dependence. The Kirchhoff term and the lack of compactness of the associated energy functional due to the Trudinger-Moser embedding have to be overcome  via some new techniques.

\vspace{3mm} \noindent {\bf Keywords:} Kirchhoff type elliptic systems; existence of solutions; multiplicity of solutions; Trudinger-Moser inequality.

\noindent \emph{\bf 2010 Mathematics Subject Classification:} 35J50, 35J57.
 \maketitle

\section{{\bfseries Introduction}}

    In this paper, we consider a class of systems of the form
    \begin{eqnarray}\label{P}
    \left\{ \arraycolsep=1.5pt
       \begin{array}{ll}
        -\big(a_1+b_1\|u\|^{2(\theta_1-1)}\big)\Delta u= \lambda H_u(x,u,v)\ \ \ &\ \mbox{in}\ \ \ \Omega,\\[2mm]
        -\big(a_2+b_2\|v\|^{2(\theta_2-1)}\big)\Delta v= \lambda H_v(x,u,v)\ \ \ &\ \mbox{in}\ \ \ \Omega,\\[2mm]
        u=0, v=0\ \ \ \ &\ \mbox{on}\ \ \ \partial\Omega,
        \end{array}
    \right.
    \end{eqnarray}
    where $\Omega$ is a bounded domain in $\mathbb{R}^2$ with smooth boundary,\ $\|w\|=\big(\int_{\Omega}|\nabla w|^2dx\big)^{1/2}$, $a_1, a_2>0$, $b_1, b_2> 0$, $\theta_1, \theta_2>1$ and $\lambda$ is a positive parameter. $H_u$ and $H_v$ have the maximal growth which allows treating (\ref{P}) variationally in the Sobolev space $H^1_0(\Omega,\mathbb{R}^2)$. Here $H^1_0(\Omega,\mathbb{R}^2)$ denotes the Sobolev space modeled in $L^2(\Omega,\mathbb{R}^2)$ with the scalar product
    \begin{equation*}\
    \langle W,\Phi\rangle=\int_{\Omega}\nabla u \cdot \nabla \varphi dx+\int_{\Omega}\nabla v \cdot \nabla \psi dx,
    \end{equation*}
    where $W=(u,v)$ and $\Phi=(\varphi,\psi)$, to which corresponds the norm $\|W\|=\langle W,W\rangle^{1/2}$. We shall consider the variational situation in which
    \begin{equation*}\
    \big(H_u(x,W),H_v(x,W)\big)=\nabla H(x,W)
    \end{equation*}
    for a function $H:\Omega \times \mathbb{R}^2 \rightarrow \mathbb{R}$ of class $C^1$, where $\nabla H$ stands for the gradient of $H$ in the variables $W=(u,v)\in \mathbb{R}^2$.

We mention that a great attention has been focused on the study of problems involving exponential growth nonlinearities, which are related to the famous Trudinger-Moser inequality. Let $\Omega$ be a bounded domain in $\mathbb{R}^2$, and denote by $H_0^{1}(\Omega)$ the standard first order Sobolev space given by
\[
H_0^{1}(\Omega)=cl\Big\{u\in C^\infty_0(\Omega)\ :\ \int_\Omega|\nabla u|^2dx<\infty\Big\},
\]
with the inner product $\langle u,v \rangle=\int_{\Omega}\nabla u \cdot \nabla v dx$ and the norm $\|u\| =\langle u,u \rangle^{1/2}$.  This space is a limiting case for the Sobolev embedding theorem, which yields $H_0^{1}(\Omega)\hookrightarrow L^p(\Omega)$ for all $1\leq p<\infty$, but one knows by easy examples that $H_0^{1}(\Omega)\not\subseteq L^\infty(\Omega)$. Hence, one is led to look for a function $g:\mathbb{R}\to\mathbb{R}^+$ with maximal growth such that
\[
\sup\limits_{u\in H_0^{1}(\Omega),\|u\| \leq 1}\int_\Omega g(u)dx<\infty.
\]
It was shown that by Trudinger \cite{Trudinger} and Moser \cite{m} that the maximal growth is of exponential type. More precisely, $\exp(\beta u^{2})\in L^1(\Omega),\ \forall\ u\in H_0^{1}(\Omega),\ \forall\ \beta>0$, and
\begin{align}\label{tmi}
\sup\limits_{u\in H_0^{1}(\Omega), \|u\| \leq 1}\int_\Omega \exp(\beta u^{2})dx< \infty,\quad \mbox{if}\  \beta\leq 4\pi,
\end{align}
where $4\pi$ is the sharp constant in the sense that the supreme in the left is $+\infty$ if  $\beta >4\pi$.
Therefore, from this result they have naturally associated notions of criticality and
subcriticality, namely, they say that a function $g: \Omega\times\mathbb{R}\to\mathbb{R}$ has subcritical growth if
\[
\lim _{|u| \rightarrow \infty} \frac{|g(x, u)|}{\exp \left(\beta u^{2}\right)}=0, \text{uniformly on } x\in\Omega,\  \forall \beta>0,
\]
and $g$ has critical growth on $\Omega$ if there exists $\beta_0>0$ such that
\begin{equation*}
    \lim _{|u| \rightarrow \infty} \frac{|g(x, u)|}{\exp \left(\beta u^{2}\right)}
    =  \begin{cases}
     0,
     & \text{uniformly on }x\in\Omega,\ \forall\beta >\beta_0,\\[3pt]
     \infty,
     & \text{uniformly on }x\in\Omega,\ \forall\beta <\beta_0.
    \end{cases}
    \end{equation*}
%\[
%\lim _{|u| \rightarrow \infty} \frac{|g(x, u)|}{\exp \left(\beta u^{2}\right)}=0, \text { uniformly on } \Omega,\  \forall \beta>\beta_0,
%\]
%and
%\[
%\lim _{|u| \rightarrow \infty} \frac{|g(x, u)|}{\exp \left(\beta u^{2}\right)}=\infty, \text { uniformly on } \Omega,\  \forall \beta<\beta_0.
%\]
We should stress that some results about the existence of solutions for elliptic problems involving exponential nonlinearities has been obtained by Adimurthi, de Figueiredo-do \'{O}-Ruf, de Figueiredo-Miyagaki-Ruf, do \'{O}, Lam-Lu, etc \cite{a,ddr,dmr1,do,lamlu1,lamlu2,lamlu3,lamlu4} and the references therein.
Moreover, de Souza \cite{ds} has studied the existence of solutions for a singular class of elliptic systems involving critical exponential growth in $\mathbb{R}^2$.

On the other hand,  system (\ref{P}) is related to the stationary analogue of the equation
\begin{eqnarray}\label{kirch}
u_{tt}-\left(a+b\int_\Omega|\nabla u|^2dx\right)\triangle u=g(x,u)
\end{eqnarray}
proposed by Kirchhoff \cite{kir} as an extension of the classical D'Alembert's wave equation for free
vibrations of elastic strings. Kirchhoff's model takes into account the changes in length of the
string produced by transverse vibrations. Equation (\ref{kirch}) received much attention only after Lions \cite{lions} proposed an abstract framework to the problem.
Equation (\ref{kirch}) is called nonlocal because of the term  $a+b\int_\Omega|\nabla u|^2dx$ which implies that
the equation in (\ref{kirch}) is no longer a pointwise identity. The presence
of the term $a+b\int_\Omega|\nabla u|^2dx$ provokes some mathematical difficulties which makes the study
of such a class of problems particularly interesting.
We refer to \cite{alves,chen,cy2,chengwu,fiscellaValdinoci,fs,hezou,hezou2,lilishi,nt,pereea1} for more existence results of the Kirchhoff type equations.
In particular, let $\Omega\subset \mathbb{R}^2$ be a bounded domain, Figueiredo and Severo \cite{fs} studied the existence of positive ground state solutions for a nonlocal Kirchhoff problem of the type
    \begin{eqnarray*}
    \left\{ \arraycolsep=1.5pt
       \begin{array}{ll}
        -m\left(\int_\Omega|\nabla u|^2 dx\right)\Delta u=f(x,u),\ \ & \mbox{in}\ \ \Omega, \\[2mm]
        u=0,\ \ & \mbox{on}\ \ \partial\Omega,
        \end{array}
    \right.
    \end{eqnarray*}
    where $m$ is a Kichhoff type function. Naimen and Tarsi  \cite{nt}  studied the multiple solutions for the following problem
    \begin{eqnarray*}
    \left\{ \arraycolsep=1.5pt
       \begin{array}{ll}
        -\big(1+b\int_{\Omega}|\nabla u|^2 dx\big)\Delta u=\lambda f(x,u),\ u\geq 0\ \ & \mbox{in}\ \ \Omega, \\[2mm]
        u=0,\ \ & \mbox{on}\ \ \partial\Omega,
        \end{array}
    \right.
    \end{eqnarray*}
where $b,\ \lambda$ are positive parameters, and the authors adjusted $b$ and $\lambda$ to get the multiple solutions. Arora etc. \cite{AGMS20} studied a system of Kichhoff equations with exponential nonlinearity of Choquard type, and de Albuquerque etc. \cite{ddds21} have studied a class of Kirchhoff systems involving critical growth in the whole $\mathbb{R}^2$. Recently, mainly inspired by \cite{ds,fs,nt}, we have studied system (\ref{P}) on Heisenberg group for some special Kirchhoff functions \cite{dt21}, and see also \cite{dt22,Ti22} for system (\ref{P}) with $\lambda=1$ and abstract Kichhoff type functions.

%Naimen and Tarsi \cite{nt} considered the multiple solutions of a Kirchhoff type elliptic problem which is different from the ground state solution, the authors adjusted the conditions $H$ to get a solution which delights us to study the existence of solutions for elliptic systems for Kirchhoff problems rather than the ground state solution.%, in which we will state Theorems \ref{thmse} and \ref{thmce} as follows.

Motivated by the above results, in the present paper, we will first prove the existence of solutions for a Kirchhoff type elliptic system (\ref{P}), and we will treat both the subcritical case and the critical case, which
we define next. Let $h(x,W)$ be a function in $\overline{\Omega} \times\mathbb{R}^2$. We say that

$(TM)^{sc}$:
    $h$ has subcritical growth at $\infty$ if
    \begin{equation*}
    \lim_{|W|\to\infty}\frac{|h(x,W)|}{\exp(\beta |W|^2)}=0,
     \ \ \text{uniformly on }x\in\Omega,\ \forall\beta >0.
    \end{equation*}

$(TM)^c$:
    $h$ has critical growth at $\infty$ if there exists $\beta_0 >0$ such that
    \begin{equation*}
    \lim_{|W|\to\infty}\frac{|h(x,W)|}{\exp(\beta |W|^2)}
    =  \begin{cases}
     0,
     & \text{uniformly on }x\in\Omega,\ \forall\beta >\beta_0,\\[3pt]
     \infty,
     & \text{uniformly on }x\in\Omega,\ \forall\beta <\beta_0.
    \end{cases}
    \end{equation*}

    Let us introduce the precise assumptions under which our problem is studied. Set
     \begin{align*}
     \theta_{max}= & \max\{\theta_1,\theta_2\},\ \theta_{min}=\min\{\theta_1,\theta_2\}>1,\\
     a_{max}= & \max\{a_1,a_2\}, \ a_{min}=\min\{a_1,a_2\}>0,\\
     b_{max}= & \max\{b_1,b_2\}, \ b_{min}=\min\{b_1,b_2\}>0.
    \end{align*}
     We give some assumptions on  $H$ as follows:
    \begin{itemize}
    \item[($H_1$):]
    $H(x,\cdot,\cdot)\in C^1(\mathbb{R}^2)$, $H(x,W)\geq 0$, $uH_u(x,W)\geq 0$, $vH_v(x,W)\geq 0$ for all $W=(u,v)\in \mathbb{R}^2$, uniformly on $x\in \Omega$.

    \item[($H_2$):]
    there exist constants $S_0,\ M_0>0$ such that

     $0<H(x,W)\leq M_0|\nabla H(x,W)|$, \ for all \ $|W|\geq S_0$ uniformly on $x\in\Omega$.

    \item[($\overline{H_2}$):]
    there exist $\mu\geq 2\theta_{max}$ and $R>0$  such that
    $0<\mu H(x,W)\leq W\cdot\nabla H(x,W)$, for all $(x,|W|)\in \Omega\times [R,+\infty)$.

    \item[($SL$):]
    $\limsup_{|W|\rightarrow 0} \frac {H(x,W)}{|W|^2}=0$ uniformly on $x\in\Omega$; superlinearity at 0.

    \item[($AL$):]
    $\limsup_{|W|\rightarrow 0} \frac {H(x,W)}{|W|^2}=\frac{a_{min}}{2}$ uniformly on $x\in\Omega$; asymptotically linearity at 0.
    \end{itemize}

    \begin{remark}\label{remfc}\rm
    We observe that condition $(H_2)$ implies that for each $\xi>0$, there exists $R_\xi>0$ such that
    \begin{equation}\label{rf3ar}
    0<\xi H(x,W)\leq |W||\nabla H(x,W)|,\ \ \forall(x,|W|)\in \Omega\times [R_\xi,+\infty),
    \end{equation}
    then for some $C_1, C_2>0$, it holds that
    \begin{equation}\label{rf3arb}
    H(x,W)\geq C_1|W|^\xi-C_2,\ \ \forall(x,W)\in \Omega\times \mathbb{R}^2.
    \end{equation}
    Furthermore, condition $(SL)$ or $(AL)$ implies $H(x,0,0)=0$ uniformly on $x\in\Omega$.
    It is worth to mention that condition $(\overline{H_2})$ is slightly stronger than (\ref{rf3ar}).
    \end{remark}

    \vspace{0.25cm}
    We say that $W=(u,v)\in H^1_0(\Omega,\mathbb{R}^2)$ is a weak solution of problem (\ref{P}) if %the following holds
    \begin{equation*}\
    \big(a_1+b_1\|u\|^{2(\theta_1-1)}\big)\int_{\Omega} \nabla u \cdot \nabla\varphi dx+\big(a_2+b_2\|v\|^{2(\theta_2-1)}\big)\int_{\Omega} \nabla v \cdot \nabla\psi dx= \lambda \int_{\Omega}\Phi\cdot \nabla H(x,W)dx,
    \end{equation*}
    for all $\Phi=(\varphi,\psi)\in H^1_0(\Omega,\mathbb{R}^2)$. %Since $H_u(x,0,0)=H_v(x,0,0)=0$, $W=(0,0)$ is the trivial solution of problem (\ref{P}). Thus, our aim is to obtain nontrivial solutions.
    %Now, the main results of this work can be stated as follows.
    We first give the results about the existence of solutions for both subcritical case and critical case:
    %\begin{theorem}\label{thmse}
    %{\rm (existence result for the subcritical case)}
    %Suppose that $H_u,\ H_v$ have subcritical growth $(TM)^{sc}$, and $(H_1)-(H_2), (AL)$ are satisfied. Then, problem (\ref{P}) has at least one nontrivial solution for each $0<\lambda<\lambda_1$, where
    %\begin{equation}\label{defl1}
    %\lambda_1:=\inf_{W\in H^1_0(\Omega,\mathbb{R}^2)\setminus\{(0,0)\}} \frac {\|W\|^2}{\int_{\Omega}|W|^2dx}>0.
    %\end{equation}
    %\end{theorem}

    \begin{theorem}\label{thmsesl}
    %{\rm (existence result for the subcritical case)}
    Suppose that $H_u,\ H_v$ have subcritical growth $(TM)^{sc}$, and $(H_1)$, $(H_2)$, $(\overline{H_2})$, $(SL)$ are satisfied. Then, problem (\ref{P}) has at least one nontrivial solution for each $\lambda>0$.
    \end{theorem}

    \begin{theorem}\label{thmcerd}
    %{\rm (existence results for the critical case)}
    Suppose that  $H_u,\ H_v$ have critical growth $(TM)^{c}$, and $(H_1)$, $(H_2), (SL)$ are satisfied. Moreover, if
    \begin{itemize}
    \item[($H_3$):]
    $0<2\theta_{max}H(x,W)\leq W\cdot\nabla H(x,W)$, for all $(x,W)\in \Omega\times\mathbb{R}^2\backslash \{(0,0)\}$,
    \end{itemize}
    then there exists a constant $\lambda_*>0$ such that problem (\ref{P}) admits at least one nontrivial solution with positive energy for all $\lambda\geq \lambda_*$.
    \end{theorem}

Next, we consider the multiplicity of solutions for problem (\ref{P}).
Recently,  Fiscella, Pucci and Zhang in \cite{fpz} obtained the existence of solution for a $p$-fractional Hardy-Schr\"{o}dinger-Kirchhoff system in $\mathbb{R}^N$ with critical nonlinearities as follows
\begin{eqnarray*}
    \left\{ \arraycolsep=1.5pt
       \begin{array}{ll}
        M(\|(u,v)\|^p)(\mathcal{L}^s_p u+V(x)|u|^{p-2}u)-\sigma \frac{|u|^{p-2}u}{|x|^{ps}}=\lambda H_u(x,u,v)+\frac{\psi}{p^*_s}|v|^\omega|u|^{\psi-2}u, \\[2mm]
        M(\|(u,v)\|^p)(\mathcal{L}^s_p v+V(x)|v|^{p-2}v)-\sigma \frac{|v|^{p-2}v}{|x|^{ps}}=\lambda H_v(x,u,v)+\frac{\omega}{p^*_s}|u|^\psi|v|^{\omega-2}v,
        \end{array}
    \right.
    \end{eqnarray*}
where $0<s<1<p<\infty,\ sp<N,\ \psi>1$ and $\omega>1$ with $\psi+\omega=p^*_s$ and $p^*_s=\frac{Np}{N-sp}$, and the authors adjusted $\lambda$ (for $\lambda$ large enough) to get the existence of solution.

Motivated by the works mentioned above and also \cite{nt}, the main purpose of the second part of this paper is to investigate the multiplicity of solutions for subcritical case and another type solution for critical case which is different from the Mountain pass solution. As stated in Lemma \ref{lemgc3}, we will define $\Lambda_0(b_{min})$ (see (\ref{deflbm})) depending on $b_{min}$ such that $\Lambda_0(b_{min})\rightarrow \infty$ as $b_{min}\rightarrow \infty$. Then there exists a value $b_0>0$ (see (\ref{defb0})) such that for all $b_{min}> b_0$, it holds that $\Lambda_0(b_{min})>\frac{a_{max}}{a_{min}}\lambda_1$, where
    \begin{equation}\label{defl1}
    \lambda_1:=\inf_{W\in H^1_0(\Omega,\mathbb{R}^2)\setminus\{(0,0)\}} \frac {\|W\|^2}{\int_{\Omega}|W|^2dx}>0.
    \end{equation}

    \begin{theorem}\label{thmscerm}
    %{\rm (multiplicity results for the subcritical case)}
    Suppose that  $H_u,\ H_v$ have subcritical growth $(TM)^{sc}$, and $(H_1)$, $(H_2)$, $(\overline{H_2})$, $(AL)$ are satisfied. Then when $b_{min}>b_0$, problem (\ref{P}) has at least two nontrivial solutions for each $\frac{a_{max}}{a_{min}}\lambda_1<\lambda<\Lambda_0(b_{min})$.
    \end{theorem}

    \begin{theorem}\label{thmcerm}
    Suppose that  $H_u,\ H_v$ have critical growth $(TM)^{c}$, and $(H_1)$, $(H_2), (AL)$ are satisfied. Then when $b_{min}>b_0$, problem (\ref{P}) has at least one nontrivial solution with negative energy for each $\frac{a_{max}}{a_{min}}\lambda_1<\lambda<\Lambda_0(b_{min})$.
    \end{theorem}

    Next, we make some remarks about above theorems.

    \begin{remark}\label{remympl}\rm
    Theorem \ref{thmcerd}, we can quantify the number $\lambda_*$ for some special examples. Such as for $H(x,W)=|W|^{\mu}\exp(\beta_0|W|^2)$ with $\mu\geq 2\theta_{max}$ and $\beta_0>0$. We fix a function $W_0=(u_0,v_0)\in H^1_0(\Omega,\mathbf{R}^2)$ satisfying $\|W_0\|=1$ such that
    \[S_\mu:=\inf_{W\in H^1_0(\Omega,\mathbf{R}^2)\setminus\{(0,0)\}} \frac {\|W\|}{(\int_{\Omega}|W|^\mu dx)^{1/\mu}}=\left(\int_{\Omega}|W_0|^\mu dx\right)^{-1/\mu}.\]
    Therefore the mountain pass level $c_{M,\lambda}$ satisfies
    \[c_{M,\lambda}\leq \max_{t\geq 0}\left[\frac{a_{max}}{2}t^2+\frac{b_{max}}{2\theta_{min}}\left(t^{2\theta_1}\|u_0\|^{2\theta_1}+t^{2\theta_2}\|v_0\|^{2\theta_2}\right)-\lambda t^\mu S^{-\mu}_\mu\right].\]
    It is enough to take
    \[\lambda_*=\max\left\{\frac{2a_{max}}{\mu S^{-\mu}_\mu}\left(\frac{b_{max}}{a_{max}\theta_{min}}\right)^{\frac{\mu-2}{2\theta_{max-2}}},
    \frac{2a_{max}}{\mu S^{-\mu}_\mu},\Gamma_*\right\}>0\]
    where
    \[\Gamma_*=\left\{\frac{\min\bigg\{\frac{1}{2}\Big(a_1\frac{2\pi}{\beta_0}+\frac{b_1}{\theta_1}\big(\frac{2\pi}{\beta_0}\big)^{\theta_1}\Big),
    \frac{1}{2}\Big(a_2\frac{2\pi}{\beta_0}+\frac{b_2}{\theta_2}\big(\frac{2\pi}{\beta_0}\big)^{\theta_2}\Big)\bigg\}}{a_{max}\left(1-\frac{2}{\mu}\right)
    \left[\frac{2a_{max}}{\mu S^{-\mu}_\mu}\right]^{\frac{2}{\mu-2}}}\right\}^{\frac{2}{\mu-2}},\]
    such that for each $\lambda>\lambda_*$, it holds that
    \[c_{M,\lambda}<\min\bigg\{\frac{1}{2}\Big(a_1\frac{2\pi}{\beta_0}+\frac{b_1}{\theta_1}\big(\frac{2\pi}{\beta_0}\big)^{\theta_1}\Big),
    \frac{1}{2}\Big(a_2\frac{2\pi}{\beta_0}+\frac{b_2}{\theta_2}\big(\frac{2\pi}{\beta_0}\big)^{\theta_2}\Big)\bigg\}.\]
    \end{remark}

\begin{remark}\label{remcer}\rm
    In Theorem \ref{thmcerm}, the reason why we couldn't get another solution with positive energy as in Theorem \ref{thmscerm}, is the difficulty of keeping the mountain pass value less than some proper level. In fact, by the standard assumption
    \begin{itemize}
     \item[($H_4$):]
    there exists $\eta>0$ such that
    $$
    \liminf_{|W|\rightarrow \infty}\frac {W\cdot\nabla H_u(x,W)}{\exp(\beta_0 |W|^2)}\geq \eta>\max\left\{\frac {4(a_1+b_1\big(\frac{4\pi}{\beta_0}\big)^{\theta_1-1})}{\beta_0 d^2 e },\frac {4(a_2+b_2\big(\frac{4\pi}{\beta_0}\big)^{\theta_2-1})}{\beta_0 d^2 e }\right\},
    $$
    \end{itemize}
    then by using some proper functions (see \cite{ddr,dt22}), we can prove the mountain pass level $c_M$ satisfying
    \begin{align*}
    \begin{split}
    c_{M}<\min\bigg\{\frac{1}{2}\Big(a_1\frac{4\pi}{\beta_0}+\frac{b_1}{\theta_1}\big(\frac{4\pi}{\beta_0}\big)^{\theta_1}\Big),
    \frac{1}{2}\Big(a_2\frac{4\pi}{\beta_0}+\frac{b_2}{\theta_2}\big(\frac{4\pi}{\beta_0}\big)^{\theta_2}\Big)\bigg\}.
    \end{split}
    \end{align*}
    However, it is not enough to make sure that functional satisfies Palais-Smale condition by Lions' concentration-compactness lemma, see (\ref{mpllp}) in Lemma \ref{lemcpsc}.
    \end{remark}

    \begin{remark}\label{remext}\rm
    There are many functions satisfying conditions in each theorem:
    \begin{itemize}
    \item[($i$)]
    in Theorem \ref{thmsesl}: $H(x,W)=|W|^{\mu}[\exp(\alpha_0|W|)-1]$ with $\mu\geq 2$ and $\alpha_0>0$  which satisfies $(\overline{H_2})$ but does not satisfy $({H_3})$ for $2\leq \mu<2\theta_{max}$;
    \item[($ii$)]
     in Theorem \ref{thmcerd}: $H(x,W)=\frac{|W|^{\mu}\exp(\beta_0|W|^2)}{2^{\alpha_0|W|^2}}$ with $\mu\geq 2\theta_{max}$ and $0\leq \alpha_0<\frac{\beta_0}{\log 2}$;
    \item[($iii$)]
    in Theorem \ref{thmscerm}: $H(x,W)=\frac{a_{min}}{2}|W|^2 \exp(\alpha_0|W|)$ with $\alpha_0>0$ which satisfies $(\overline{H_2})$ but does not satisfy $({H_3})$;
    \item[($iv$)]
    in Theorem \ref{thmcerm}: $H(x,W)=\frac{a_{min}}{2}|W|^2\exp(\beta_0|W|^2)$ with $\beta_0>0$ which does not satisfy $({H_3})$.
    \end{itemize}
    Moreover, it is worth noting that if we add conditions in $(H_1)$ that $H_u(x,u,v)\equiv 0$ and $H_v(x,u,v)\equiv 0$ for $u\leq 0$ or $v\leq 0$, uniformly on $x\in\Omega$, then by Maximum principle, we can proof that the solutions we obtain in above theorems are entire positive, i.e. $u>0$ and $v>0$ in $\Omega$, and the example functions only need to be defined as $\overline{H}(x,W)=H(x,W)\mbox{sign}u\ \mbox{sign} v$, where $\mbox{sign}t=0$ for $t\leq 0$ and $\mbox{sign}t=1$ for $t>0$.
    \end{remark}

   % \begin{remark}\rm
    %In (\cite{nt}, Remark 1.8 and Theorems 1.9,\ 1.10), the authors weakened the Ambrosetti-Rabinowitz condition, however, for our problem (\ref{P}), we can't weaken $(H_3)$ because of the difficulties for systems.
    %\end{remark}

    %To the best of our knowledge, this is the first work which shows the existence of multiple solutions introduced by nonlocal coefficient for the problem in a dual elliptic system with Trudinger-Moser growth.

    This paper is organized as follows: Section \ref{sec preliminaries} contains some technical results. In Section \ref{vf}, we present the variational setting in which our problem will be treated. Section \ref{ps} is devoted to showing some properties of the Palais-Samle sequences. Then, we split Section \ref{main} into two subsections for the subcritical case and critical case, and we complete the proofs of Theorems \ref{thmsesl}-\ref{thmcerd}. And then, we will study the multiple solutions of problem (\ref{P}). In section \ref{vfh6}, we construct an important variational framework and complete the proofs of Theorems \ref{thmcerd}-\ref{thmscerm}. Hereafter, $C,C_0,C_1,C_2...$ will denote positive (possibly different) constants.

\section{{\bfseries Some preliminary results}}\label{sec preliminaries}

    We first introduce some famous inequalities as follows, and inspired by those inequalities, we conclude some similar forms of inequalities.

    \begin{lemma}\label{lemtm1}
    \cite{m} Let $\Omega$ be a bounded domain in $\mathbb{R}^2$ and $u\in H^1_0(\Omega)$. Then for every $\beta >0$,
    \begin{align*}\
    \int_{\Omega} \exp(\beta |u|^2)dx<\infty.
    \end{align*}
    Moreover, there exists a constant $C(\Omega)>0$ such that
    $$
    \sup_{u\in H^1_0(\Omega):\|u\| \leq 1}\int_{\Omega}\exp(\beta |u|^2)dx\leq C(\Omega),
    \ \ \forall\beta \leq 4\pi .
    $$
   Furthermore,\  $4\pi$ is the best constant, that is, the supreme in the left is $\infty$ if \ $\beta >4\pi$.
    \end{lemma}

    \begin{lemma}\label{lemtm3}
    \cite{l} Let $\{u_n\}$ be a sequence of functions in $H^1_0(\Omega)$ with $\|\nabla u_n\|_2=1$ such that $u_n\rightharpoonup u\neq0$ weakly in  $H^1_0(\Omega)$. Then for any  $0<p<4\pi /(1-\|\nabla u\|^2_2)$ we have
    $$
    \sup_n \int_{\Omega}\exp(p|u_n|^2)dx<\infty.
    $$
    \end{lemma} \ \

    Now, by the above results, we begin to establish the Trudinger-Moser type inequalities in $H^1_0(\Omega,\mathbb{R}^2)$:
    \begin{lemma}\label{lemtm4}
    Let $\Omega$ be a bounded domain in $\mathbb{R}^2$ and $(u,v)\in H^1_0(\Omega,\mathbb{R}^2)$. Then for every $\beta >0$,
    \begin{align}\label{tmt}
    \int_{\Omega} \exp(\beta |(u,v)|^2)dx<\infty.
    \end{align}
    Moreover, it holds that
    \begin{equation}\label{tmtr}
    \sup_{(u,v)\in H^1_0(\Omega,\mathbb{R}^2), \|(u,v)\|\leq 1}\int_{\Omega}\exp\left(\beta|(u,v)|^2\right) dx \leq C(\Omega),\ \ \forall \beta\leq 4\pi,
    \end{equation}
    where $C(\Omega)$ is given as in Lemma \ref{lemtm1}.
    \end{lemma}

    \begin{proof}
    For each $W=(u,v)\in H^1_0(\Omega,\mathbb{R}^2)$, we have that $|W|^2=|u|^2+|v|^2$. Then by using H\"{o}lder's inequality and Lemma \ref{lemtm1},
    \begin{align*}
    \begin{split}
    \int_{\Omega} \exp(\beta |W|^2)dx
    = & \int_{\Omega} \exp(\beta |u|^2) \exp(\beta |v|^2)dx \\
    \leq & \left(\int_{\Omega}\exp(2\beta |u|^2)dx\right )^{1/2} \left(\int_{\Omega} \exp(2\beta |v|^2)dx\right)^{1/2}<\infty.
    \end{split}
    \end{align*}
    Thus (\ref{tmt}) holds.

    Then we begin to prove (\ref{tmtr}). It is sufficiently enough to prove that
    \begin{equation}\label{aie}
    \sup_{(u,v)\in H^1_0(\Omega,\mathbb{R}^2), \|(u,v)\|=1}\int_{\Omega}\exp\left(4\pi|(u,v)|^2\right) d x \leq C(\Omega).
    \end{equation}
    Indeed, for each $(u,v)\in H^1_0(\Omega,\mathbb{R}^2)$ with $\|(u,v)\|=1$, then $\|(u,v)\|^2=\|u\|^2+\|v\|^2=1$. If $\|u\|=0$ or $\|v\|=0$, then from (\ref{lemtm1}), we know (\ref{aie}) holds also. Otherwise, $\|u\|\neq0$ and $\|v\|\neq0$, then we take $\frac{1}{p}=\|u\|^2$, $\frac{1}{q}=\|v\|^2$ and by the Young's inequality, we have
    \begin{equation*}
    \begin{split}
    \int_{\Omega}\exp\left(4\pi|(u,v)|^2\right) dx
    = & \int_{\Omega}\exp\left(4\pi u^2\right) \exp\left(4\pi v^2\right) dx  \\
    \leq & \int_{\Omega}\frac{\exp\left(4\pi pu^2\right)}{p}dx
    +\int_{\Omega} \frac{\exp\left(4\pi qv^2\right)}{q} dx  \\
    = & \|u\|^2\int_{\Omega}\exp\left(4\pi\left(\frac{u}{\|u\|}\right)^2\right)dx
    +\|v\|^2\int_{\Omega}\exp\left(4\pi\left(\frac{v}{\|v\|}\right)^2\right)dx  \\
    \leq & \|u\|^2C(\Omega)+\|v\|^2C(\Omega)=C(\Omega),
    \end{split}
    \end{equation*}
    thus (\ref{aie}) holds and the proof is complete.
    \end{proof}

    \begin{lemma}\label{lemtm5}
    Let $\{W_n\}$ be a sequence of functions in $H^1_0(\Omega,\mathbb{R}^2)$ with $\| W_n\|=1$ such that $W_n\rightharpoonup W_0\neq(0,0)$ weakly in  $H^1_0(\Omega,\mathbb{R}^2)$. Then for any $0<p<4\pi /(1-\| W_0\|^2)$, we have
    $$
    \sup_n \int_{\Omega}\exp(p|W_n|^2)dx<\infty.
    $$
    \end{lemma}

    \begin{proof}
    Since $W_n\rightharpoonup W_0$ and $\|W_n\|=1$, we conclude that
    \begin{align*}
    \|W_n-W_0\|^2=1-2\langle W_n,W_0\rangle+\|W_0\|^2\rightarrow 1-\|W_0\|^2<\frac {4\pi}{p}
    \end{align*}
    Thus, for large $n$ we have
    \begin{align*}
    p\|W_n-W_0\|^2<4\pi.
    \end{align*}
    Now choosing $q>1$ close to 1 and $\epsilon>0$ satisfying
    \begin{align*}
    qp(1+\epsilon^2)\|W_n-W_0\|^2\leq 4\pi.
    \end{align*}
    Then Lemma \ref{lemtm4} indicates that
    \begin{align}\label{tmlh}
    \int_{\Omega}\exp\left(qp(1+\epsilon^2)|W_n-W_0|^2\right)dx\leq C(\Omega).
    \end{align}
    Moreover, since
    \begin{align*}
    p|W_n|^2 \leq p (1+ \epsilon ^2)|W_n-W_0|^2 + p(1+1/\epsilon ^2)|W_0|^2
    \end{align*}
    which can be proved by  Young's inequality, it follows that
    \begin{align*}
    \exp(p|W_n|^2)\leq \exp(p (1+ \epsilon ^2)|W_n-W_0|^2) \exp(p(1+1/\epsilon ^2)|W_0|^2).
    \end{align*}
    Consequently, by  H\"{o}lder's inequality and (\ref{tmlh}), it holds that
    \begin{align*}
    \int_{\Omega}\exp(p|W_n|^2)dx
    \leq & \left (\int_{\Omega}\exp(qp(1+\epsilon^2)|W_n-W_0|^2)dx\right)^{\frac{1}{q}}\left(\int_{\Omega}\exp(rp(1+1/\epsilon^2)|W_0|^2)dx\right)^{\frac{1}{r}}\\
    \leq & C\left (\int_{\Omega}\exp(rp(1+1/\epsilon^2)|W_0|^2)dx\right)^{\frac{1}{r}},
    \end{align*}
    for large $n$, where $r=\frac{q}{q-1}$. By Lemma \ref{lemtm4}, we know the second term in the last inequality is bounded, thus the proof is complete.
    \end{proof}

    \begin{lemma}\label{lemiiw}
    If $W=(u,v)\in H^1_0(\Omega,\mathbb{R}^2),\beta>0,q>0$ and $\|W\|\leq N$ with $\beta N^2<2\pi$, there exists $C=C(\beta,N,q)>0$ such that
    \begin{align}\label{asd1}
    \int_{\Omega} |W|^q \exp(\beta |W|^2)dx\leq C\|W\|^q.
    \end{align}
    \end{lemma}

    \begin{proof}
    We consider $r> 1$ close to 1 such that $r \beta N^2\leq2 \pi$ and $sq\geq 1$, where $s=\frac {r}{r-1}$. Using H\"{o}lder's inequality and Young's inequality, we have
    \begin{align*}
    \begin{split}
    \int_{\Omega} |W|^q \exp(\beta |W|^2)dx
    \leq & \left (\int_{\Omega}\exp(r\beta |W|^2)dx\right)^{\frac{1}{r}}\|W\|^q_{qs} \\
    \leq & \left (\int_{\Omega}\exp(2r\beta N^2 (\frac{u}{\|W\|})^2)dx+\int_{\Omega}\exp(2r\beta N^2 (\frac{v}{\|W\|})^2)dx\right)^{\frac{1}{r}}\|W\|^q_{qs}
    \end{split}
    \end{align*}
    Finally, using the continuous embedding $H^1_0(\Omega,\mathbb{R}^2)\hookrightarrow L^{sq}(\Omega,\mathbb{R}^2)$ and Lemma \ref{lemtm1}, we conclude that
    \begin{align*}
    \int_{\Omega} |W|^q \exp(\beta |W|^2)dx\leq C\|W\|^q.
    \end{align*}
    \end{proof}

    Finally, in order to deal with the dual Kirchhoff functions, in Section \ref{vfh6} we need the following algebraic inequality:
    \begin{lemma}\label{lemiit}
    If $x,y\geq 0,\ t\geq 1$, then
    \begin{align}\label{iitk}
   % \frac{1}{2^t-1}(x+y)^t\leq x^t+y^t \leq(x+y)^t.
    (x+y)^t\leq (2^t-1)(x^t+y^t) \leq(2^t-1)(x+y)^t.
    \end{align}
    \end{lemma}

    \begin{proof}
    It is obvious that $x^t+y^t \leq(x+y)^t$. If $0\leq y\leq x$, then we fix $y$ and define $g(x)=(x^t+y^t)-\frac{(x+y)^t}{2^t-1}$. It holds that
    \begin{align*}
    \begin{split}
    g(y)= & 2y^t-\frac{2^ty^t}{2^t-1}=\frac{2^t-2}{2^t-1}y^t\geq 0,\\
    g'(x)= & tx^{t-1}-\frac{t(x+y)^{t-1}}{2^t-1}\geq \frac{2^{t-1}-1}{2^t-1}(t x^{t-1})\geq 0,
    \end{split}
    \end{align*}
    thus $g(x)\geq 0$ for all $x\geq y$, i.e. $(x+y)^t\leq (2^t-1)(x^t+y^t)$. On the other hand, if $x\leq y$, the result can be obtained by the same way.
     The proof is complete.
    \end{proof}

\section{{\bfseries The variational framework}}\label{vf}
    We now consider the functional $J_\lambda$ given by
    \begin{equation*}
    J_\lambda(W)=\frac {1}{2}\Big(a_1\|u\|^2+\frac{b_1}{\theta_1}\|u\|^{2\theta_1}\Big)+\frac {1}{2}\Big(a_2\|v\|^2+\frac{b_2}{\theta_2}\|v\|^{2\theta_2}\Big)-\lambda\int_{\Omega}H(x,W)dx,
    \end{equation*}
    where $W=(u,v)\in H^1_0(\Omega,\mathbb{R}^2)$.

    \begin{lemma}\label{ic1}
    Under our assumptions, we have that $J_\lambda$ is well defined and $C^1$ on $H^1_0(\Omega,\mathbb{R}^2)$ for all $\lambda>0$. Moreover,
    \begin{equation*}
    \begin{split}
    \langle J_\lambda'(W),\Phi\rangle=&\big(a_1+b_1\|u\|^{2(\theta_1-1)}\big)\int_{\Omega} \nabla u \cdot \nabla\varphi dx+\big(a_2+b_2\|v\|^{2(\theta_2-1)}\big)\int_{\Omega} \nabla v \cdot \nabla\psi dx \\
    &-\lambda\int_{\Omega}\Phi \cdot \nabla H(x,W)dx,
    \end{split}
    \end{equation*}
    where $\Phi=(\varphi,\psi)\in H^1_0(\Omega,\mathbb{R}^2)$.
    \end{lemma}

 \begin{proof}
 We have that $H_u$ and $H_v$ are continuous and have subcritical (or critical) growth $(TM)^{sc}$ (or $(TM)^{c}$). Thus, giving $\beta>0$ (or $\beta>\beta_0$), there exists $C>0$ such that $|H_u(x,W)|,\ |H_v(x,W)|\leq Ce^{\beta |W|^2}$ for all $(x,W)\in\Omega \times \mathbb{R}^2$. Then,
     \begin{equation} \label{3.1}
    |\nabla H(x,W)|\leq |H_u(x,W)|+|H_v(x,W)|\leq C_1 \exp(\beta |W|^2),
     \ \ \text{for all }(x,W)\in\Omega \times \mathbb{R}^2.
    \end{equation}
By $(SL)$ (or $(AL)$), given $\epsilon>0$, there exists $\delta>0$ such that
    \begin{equation}\label{3.2}
    |H(x,W)|\leq \frac { a_{min}+\epsilon}{2} |W|^2
     \ \ \text{always that }|W|<\delta.
    \end{equation}
Thus, using (\ref{3.1}), (\ref{3.2}) and $(H_3)$, we have
    \begin{align*}
    \int_{\Omega}|H(x,W)|dx \leq \frac { a_{min}+\epsilon}{2}\int_{\Omega}|W|^2 dx + C_1\int_{\Omega}|W|\exp(\beta |W|^2)dx.
    \end{align*}
Considering the continuous imbedding $H^1_0(\Omega,\mathbb{R}^2)\hookrightarrow L^s(\Omega,\mathbb{R}^2)$ for $s\geq 1$ and using Lemma \ref{lemiiw}, it follows that $H(x,W)\in L^1(\Omega)$ which implies that $J_\lambda$ is well defined. Using standard arguments together with Lemma \ref{lem3.1} underlying (see \cite[Theorem A.VI]{bl} and \cite{c}), we can see that $J_\lambda\in C^1\big(H^1_0(\Omega,\mathbb{R}^2),\mathbb R\big)$ with
    \begin{equation*}
    \begin{split}
    \langle J_\lambda'(W),\Phi\rangle=&\big(a_1+b_1\|u\|^{2(\theta_1-1)}\big)\int_{\Omega} \nabla u \cdot \nabla\varphi dx+\big(a_2+b_2\|v\|^{2(\theta_2-1)}\big)\int_{\Omega} \nabla v \cdot \nabla\psi dx \\
    &-\lambda\int_{\Omega}\Phi \cdot \nabla H(x,W)dx,
    \end{split}
    \end{equation*}
    where $\Phi=(\varphi,\psi)\in H^1_0(\Omega,\mathbb{R}^2)$.
    \end{proof}

    From Lemma \ref{ic1}, we have that the critical points of the functional $J_\lambda$ are precisely the weak solutions of problem (\ref{P}).
    In the next two lemmas we check that the functional $J_\lambda$ satisfies the geometric conditions of the Mountain-pass theorem.

    \begin{lemma}\label{lemgc4}
    Suppose that $(H_2),\ (SL)$ hold and the functions $H_u,\ H_v$ have subcritical (or critical) growth $(TM)^{sc}$ (or $(TM)^{c}$). Then, for each $\lambda>0$, there exist positive number $\rho_0$ and $\tau_0$ such that
    \begin{align*}
    J_\lambda(W)\geq \tau_0,\ \forall W\in H^1_0(\Omega,\mathbb{R}^2)\ \ \mbox{with} \ \ \|W\|=\rho_0.
    \end{align*}
    \end{lemma}

    \begin{proof}
    By $(SL)$, given $0<\epsilon<\frac{\lambda_1 a_{min}} {2\lambda}$ there exists $\delta>0$ such that
    \begin{align*}
    H(x,W)\leq \epsilon|W|^2,
    \end{align*}
    always that $|W|<\delta$. On the other hand, for $\beta>0$ (subcritical case) or $\beta>\beta_0$ (critical case), we have that there exists $C_1>0$ such that $|H_u(x,W)|,\ |H_v(x,W)|\leq C_1|W|^{q-1} \exp(\beta |W|^2)$ for all $|W|\geq \delta$ with $q>2$. Thus, by $(H_2)$, we have
    \begin{align}\label{gc4.1}
    H(x,W) \leq \epsilon|W|^2+C_2|W|^q \exp(\beta |W|^2),\ \forall (x,W)\in \Omega\times \mathbb{R}^2.
    \end{align}
    Thus,
    \begin{align*}
    \begin{split}
    J_\lambda(W)=&\frac {1}{2}\Big(a_1\|u\|^2+\frac{b_1}{\theta_1}\|u\|^{2\theta_1}\Big)+\frac {1}{2}\Big(a_2\|v\|^2+\frac{b_2}{\theta_2}\|v\|^{2\theta_2}\Big)-\lambda\int_{\Omega}H(x,W)dx \\
     \geq &\frac{1}{2}a_{min}\|W\|^2-\lambda \epsilon \int_{\Omega}|W|^2dx-\lambda C_2\int_{\Omega}|W|^q \exp(\beta |W|^2)dx \\
     \geq &\Big(\frac{a_{min}}{2}-\frac {\lambda \epsilon}{\lambda_1}\Big)\|W\|^2-\lambda C_2\int_{\Omega}|W|^q \exp(\beta |W|^2)dx.
    \end{split}
    \end{align*}
    By Lemma \ref{lemiiw}, there exists $N>0$ such that $\beta N^2<2\pi$ and we take $\|W\|\leq N$, there exists $C>0$ such that $\int_{\Omega}|W|^q \exp(\beta |W|^2)dx\leq C\|W\|^q$. Therefore,
    \begin{align*}
    J_\lambda(W)\geq \Big(\frac{a_{min}}{2}-\frac {\lambda \epsilon}{\lambda_1}\Big)\|W\|^2-\lambda C_3\|W\|^q,
    \end{align*}
    Since $q>2$, there exists $\kappa_1>0$ such that $\Big(\frac{a_{min}}{2}-\frac {\lambda \epsilon}{\lambda_1}\Big)\kappa^2_1-\lambda C_3\kappa^q_1>0$.
    Consequently, taking $\rho_0=\min\{N,\kappa_1\}>0$, we can get that $J_\lambda(W)\geq\tau_0$ whenever $\|W\|=\rho_0$, where $\tau_0=\Big(\frac{a_{min}}{2}-\frac {\lambda \epsilon}{\lambda_1}\Big)\rho^2_0-\lambda C_3\rho^q_0>0$. Thus, the proof is finished.
    \end{proof}

    \begin{lemma}\label{lemgc2}
    Suppose that $(H_1)$, $(H_2)$ hold and $H_u,\ H_v$ have subcritical (or critical) growth $(TM)^{sc}$ (or $(TM)^{c}$). Then there exists $E_0\in H^1_0(\Omega,\mathbb{R}^2)$ with $\|E_0\|>\rho_0$ such that $J_\lambda(E_0)<0$.
    \end{lemma}

    \begin{proof}
%     We shall make use of the polar coordinate representation
%    \begin{align*}
%    W=(\upsilon,\phi),\ \ \mbox{where}\ \ \upsilon\geq S_0,\ \ -\pi\leq\phi\leq\pi \ \ \mbox{and}\ \ u=\upsilon \sin(\phi),\ \ v=\upsilon \cos(\phi),
%    \end{align*}
%    substituting in $(H_2)$, we get $0<H(x,W)\leq M_0H_{\upsilon}(x,W)$, where $S_0$ and $M_0$ were given in $(H_2)$. Thus, by $(H_1)$, we can conclude that there exist constants $C_1,\ C_2>0$ such that
%    \begin{align*}
%    H(x,W)\geq C_1e^{\frac{|W|}{M_0}}-C_2,\ \ \mbox{for all}\ \ (x,W)\in \Omega\times H^1_0(\Omega,\mathbb{R}^2).
%    \end{align*}
    (\ref{rf3arb}) indicates that $H(x,W)\geq C_1|W|^{\xi}-C_2$ for all $(x,W)\in \Omega\times \mathbb{R}^2$, and here we choose $\xi>2\theta_{max}$.
    Thus,
    \begin{align*}
    \begin{split}
    J_\lambda(tW)=&\frac {1}{2}\Big(a_1\|u\|^2 t^2+\frac{b_1}{\theta_1}\|u\|^{2\theta_1} t^{2\theta_1}\Big)+\frac {1}{2}\Big(a_2\|v\|^2 t^2+\frac{b_2}{\theta_2}\|v\|^{2\theta_2}t^{2\theta_2}\Big)-\lambda\int_{\Omega}H(x,W)dx \\
    \leq &\frac {1}{2}\Big(a_1\|u\|^2+a_2\|v\|^2\Big) t^2+\frac{b_1}{2\theta_1}\|u\|^{2\theta_1}t^{2\theta_1}+\frac{b_2}{2\theta_2}\|v\|^{2\theta_2} t^{2\theta_2}\\
    &-C_1t^{\xi} \int_{\Omega}|W|^{\xi} dx+C_2|\Omega|, \\
    \end{split}
    \end{align*}
    for all $t>0$, $W\in H^1_0(\Omega,\mathbb{R}^2)\backslash \{(0,0)\}$, and $|\Omega|$ is the Lebesgue measure of $\Omega$. This yields that $J_\lambda(tW)\rightarrow -\infty$ as $t\rightarrow+\infty$. Setting $E_0=tW$ with $t$ large enough such that $J_\lambda(E_0)<0$ with $\|E_0\|>\rho_0$.
    \end{proof}

\section{{\bfseries On Palais-Smale sequences}}\label{ps}

    In order to prove that a Palais-Smale sequence converges to a weak solution of problem (\ref{P}), we need to establish the following lemma:

    \begin{lemma}\label{lem4.1}
    Suppose $(H_1)$, $(\overline{H_2})$ hold and $H_u,\ H_v$ have subcritical (or critical) growth $(TM)^{sc}$ (or $(TM)^{c}$). For each $\lambda>0$, and let $\{W_n\}\ in\ H^1_0(\Omega,\mathbb{R}^2)$ be the Palais-Smale sequence for functional $J_\lambda$ at finite level. Then there exists $C>0$ such that
    \begin{align}\label{psbdd}
    \|W_n\|\leq C,\ \ \ \ \int_{\Omega}u_n H_u(x,W_n)dx\leq C, \ \ \ \ \int_{\Omega}v_n H_v(x,W_n)dx\leq C.%\ \ \ \ \mbox{and} \ \ \int_{\Omega}H(x,W_n)dx\leq C.
    \end{align}
    \end{lemma}

    \begin{proof}
    Let $\{W_n\} \subset H^1_0(\Omega,\mathbb{R}^2)$ be a sequence such that $J_\lambda(W_n)\rightarrow c$ and $J_\lambda'(W_n)\rightarrow 0$ with $|c|<\infty$, we can take this as follows:
    \begin{equation}\label{4.1}
    \begin{split}
    J_\lambda(W_n)=&\frac {1}{2}\Big(a_1\|u_n\|^2+\frac{b_1}{\theta_1}\|u_n\|^{2\theta_1}\Big)+\frac {1}{2}\Big(a_2\|v_n\|^2+\frac{b_2}{\theta_2}\|v_n\|^{2\theta_2}\Big)-\lambda\int_{\Omega}H(x,W_n)dx \\
    =&c+\delta_n,\ \
    \end{split}
    \end{equation}
    where $\delta_n\rightarrow 0$ as $n\rightarrow \infty$, and
    \begin{equation}\label{4.2}
    \begin{split}
    \langle J_\lambda'(W_n),W_n\rangle=&\big(a_1+b_1\|u_n\|^{2(\theta_1-1)}\big)\|u_n\|^2+\big(a_2+b_2\|v_n\|^{2(\theta_2-1)}\big)\|v_n\|^2 \\
    &-\lambda\int_{\Omega}W_n\cdot \nabla H(x,W_n)dx=o(\|W_n\|).
    \end{split}
    \end{equation}
    Thus for $n$ large enough, by $(\overline{H_2})$ we have
    \begin{align*}
    \begin{split}
    C_1+\|W_n\|\geq &J_\lambda(W_n)-\frac {1}{\mu} \langle J_\lambda'(W_n),W_n\rangle \\
    =&\frac {1}{2}\Big(a_1\|u_n\|^2+\frac{b_1}{\theta_1}\|u_n\|^{2\theta_1}\Big)+\frac {1}{2}\Big(a_2\|v_n\|^2+\frac{b_2}{\theta_2}\|v_n\|^{2\theta_2}\Big) \\
    &-\frac{1}{\mu}\big(a_1+b_1\|u_n\|^{2(\theta_1-1)}\big)\|u_n\|^2-\frac{1}{\mu}\big(a_2+b_2\|v_n\|^{2(\theta_2-1)}\big)\|v_n\|^2  \\
    &+\int_{\Omega}\left (\frac{1}{\mu} W_n\cdot \nabla H(x,W_n) -H(x,W_n)\right )dx \\
    \geq &\Big(\frac{1}{2}-\frac{1}{\mu}\Big)\Big(a_1\|u_n\|^2+a_2\|v_n\|^2\Big)-C_{\mu}|\Omega|  \\
    %+\Big(\frac{1}{2\theta_1}-\frac{1}{\mu}\Big)b_1\|u_n\|^{2\theta_1}
    %+\Big(\frac{1}{2\theta_2}-\frac{1}{\mu}\Big)b_2\|v_n\|^{2\theta_2} \\
    \geq&\frac {\mu-2}{2\mu}a_{min}\|W_n\|^2-C_{\mu}|\Omega|,
    \end{split}
    \end{align*}
    for some $C_1>0$, where $C_\mu=\sup\{|H(x,W)-\frac{1}{\mu} W\cdot \nabla H(x,W)|: (x,|W|)\in \overline{\Omega}\times [0,R]\}$. Then $\|W_n\|$ is bounded since $\mu\geq 2\theta_{max}>2$ and $a_{min}>0$. From (\ref{4.1}) and (\ref{4.2}), it can be concluded directly that there exists a constant $C>0$ such that (\ref{psbdd}) holds since $(H_1)$ holds. %$\|W_n\|\leq C$, $\int_{\Omega}|W_n\cdot \nabla H(x,W_n)|dx\leq C$ and $\int_{\Omega}H(x,W_n)dx\leq C$.
    \end{proof}

    \begin{lemma}\label{lem4.2}
    Suppose $\Omega$ is a bounded domain in $\mathbb{R}^2$. Let $\{u_n\}$ be in $L^1(\Omega)$ such that $u_n\rightarrow u$ in $L^1(\Omega)$ and $f(x,s)$ be a continuous function. Then $f(x,u_n)\rightarrow f(x,u)$ in  $L^1(\Omega)$ provided that $f(x,u_n)\in L^1(\Omega)$ for all $n$ and $\int_{\Omega}|f(x,u_n)u_n|dx\leq C$.
    \end{lemma}

    \begin{proof}
    This proof can be found, for instant, in \cite{dmr1}.
    \end{proof}

    Then we will use the following result which is a converse of the Lebesgue dominated convergence theorem in the space $H^1_0(\Omega)$.

    \begin{lemma}\label{lem3.1}
    Let $\{u_n\}$ be a sequence of functions in $H^1_0(\Omega)$ strongly convergent. Then there exist a subsequence $\{u_{n_k}\}$ of $\{u_n\}$ and
    $g\in H^1_0(\Omega)$ such that $|u_{n_k}(x)|\leq g(x)$ almost everywhere in $\Omega$.
    \end{lemma}

    \begin{proof}
    See the proof in \cite{dms}.
    \end{proof}

    \begin{remark}\label{rem4.3}\rm
    Under the assumptions in Lemma \ref{lem4.1}, by Lemma \ref{lem4.2}, we can get that
    \begin{equation}\label{4.3}
    H_u(x,W_n)\rightarrow H_u(x,W)\ \ \mbox{and}\ \ H_v(x,W_n)\rightarrow H_v(x,W)\ \ \mbox{in}\ \ L^1(\Omega),
    \end{equation}
    where $W$ is the weak limit of Palais-Smale sequence $\{W_n\}$.
    \end{remark}

    \begin{lemma}\label{lem4.4}
    Suppose that $(H_1)-(H_2)$, $(\overline{H_2})$ hold and $H_u,\ H_v$ have subcritical (or critical) growth $(TM)^{sc}$ (or $(TM)^{c}$). If $\{W_n\} \subset H^1_0(\Omega,\mathbb{R}^2)$ is a Palais-Smale sequence for $J_\lambda$ and $W$ is its weak limit then, up to a subsequence,
    \begin{equation}\label{4.5}
    H(x,W_n)\rightarrow H(x,W)\ \ in \ \ L^1(\Omega).
    \end{equation}
    \end{lemma}

    \begin{proof}
    From Remark \ref{rem4.3}, we have
    \begin{align*}
    H_u(x,W_n)\rightarrow H_u(x,W)\ \ \mbox{and}\ \ H_v(x,W_n)\rightarrow H_v(x,W)\ \ \mbox{in}\ \ L^1(\Omega).
    \end{align*}
    Thus, by Lemma \ref{lem3.1}, there exist $F_1,\ F_2\in L^1(\Omega)$ such that $|H_u(x,W_n)|\leq F_1(x),\ |H_v(x,W_n)|\leq F_2(x)$ almost everywhere in $\Omega$. From $(H_2)$ we can conclude that
    \begin{align*}
    \begin{split}
    |H(x,W_n)|&\leq \sup_{(x,|W_n|)\in\Omega \times [0,S_0]} |H(x,W_n(x))|+M_0|\nabla H(x,W_n)| \\
    &\leq \sup_{(x,|W_n|)\in\Omega \times [0,S_0]} |H(x,W_n(x))|+M_0(F_1+F_2)  \ \ \mbox{a.e.} \ \ \mbox{in}\ \ \Omega.
    \end{split}
    \end{align*}
    Thus, by the generalized Lebesgue dominated convergence theorem
    \begin{equation*}
    H(x,W_n)\rightarrow H(x,W)\ \ \mbox{in}\ \ L^1(\Omega).
    \end{equation*}
    \end{proof}

\section{{\bfseries Existence results}}\label{main}
    According to Lemmas \ref{lemgc4} and \ref{lemgc2}, let
    \begin{equation}\label{5.1}
    c_{M,\lambda}=\inf_{\gamma\in\Upsilon}\max_{t\in [0,1]}J_\lambda(\gamma(t))>0,
    \end{equation}
    be the minimax level for $J_\lambda$, where $\Upsilon=\{\gamma\in C\big([0,1],H^1_0(\Omega,\mathbb{R}^2)\big):\gamma(0)=0,\gamma(1)=E_0\}$, and $E_0$ is given in Lemma \ref{lemgc2}. Therefore, by using the Mountain-pass theorem, there exists a sequence $\{W_n\} \subset H^1_0(\Omega,\mathbb{R}^2)$ satisfying
    \begin{equation}\label{5.2}
    J_\lambda(W_n)\rightarrow c_{M,\lambda}\ \ \mbox{and}\ \ J_\lambda'(W_n)\rightarrow 0.
    \end{equation}

\subsection{\bfseries Subcritical case: Proof of Theorem \ref{thmsesl}}
    %{\rm (existence result for the subcritical case)}

    \begin{lemma}\label{lem5.1}
    Suppose that $H_u,\ H_v$ have subcritical growth $(TM)^{sc}$, then functional $J_\lambda$ satisfies the Palais-Smale condition at any finite level.
    \end{lemma}

    \begin{proof}
    Let $\{W_n\} \subset H^1_0(\Omega,\mathbb{R}^2)$ be a sequence such that $J_\lambda(W_n)\rightarrow c$ and $J_\lambda'(W_n)\rightarrow 0$ with $|c|<\infty$. From Lemma \ref{lem4.1}, we have $\{W_n\}$ is bounded and there exists $C_0>0$ such that $\|W_n\|\leq C_0$, so we can get a subsequence still denoted by $\{W_n\}$, for some $W=(u,v)\in H^1_0(\Omega,\mathbb{R}^2)$
    \begin{align*}
    \begin{split}
    &W_n\rightharpoonup W\ \ \mbox{in} \ \ H^1_0(\Omega,\mathbb{R}^2), \\
    &W_n \rightarrow W\ \ \mbox{in}\ \ L^s(\Omega,\mathbb{R}^2),\ \ \mbox{for all}\ \ s\geq 1.
    \end{split}
    \end{align*}
    Since
    \begin{equation}\label{5.3}
    \begin{split}
    \langle J_\lambda'(W_n),W_n-W\rangle=&\big(a_1+b_1\|u_n\|^{2(\theta_1-1)}\big)\langle u_n,u_n-u\rangle+\big(a_2+b_2\|v_n\|^{2(\theta_2-1)}\big)\langle v_n,v_n-v\rangle \\
    &-\lambda\int_{\Omega}(W_n-W)\cdot \nabla H(x,W_n)dx.
    \end{split}
    \end{equation}
    From $J_\lambda'(W_n)\rightarrow 0$ in $\big(H^1_0(\Omega,\mathbb{R}^2)\big)_*$, we have $\langle J_\lambda'(W_n),W_n-W\rangle \rightarrow 0$.
    Meanwhile, by subscritical growth condition $(TM)^{sc}$ and the H\"{o}lder's inequality, it follows that
    \begin{align*}
    \begin{split}
    \int_{\Omega}|W_n-W|\cdot|\nabla H(x,W_n)|dx &\leq C\int_{\Omega}|W_n-W|\exp(\beta |W_n|^2)dx \\
    &\leq C\|W_n-W\|_2\left(\int_{\Omega}\exp(2\beta |W_n|^2)dx\right)^\frac {1}{2},
    \end{split}
    \end{align*}
    for all $\beta>0$, $r>1$, and $s= \frac {r}{r-1}$. Choosing $\beta=\frac{2\pi}{C^2_0}$, we have $2\beta \|W_n\|^2\leq 4\pi$, then by Lemma \ref{lemtm4}  and the embedding $H^1_0(\Omega,\mathbb{R}^2)\hookrightarrow L^s(\Omega,\mathbb{R}^2)$ is compact for $s\geq 2$, the second term in the last inequality converges to zero.
    Thus, in (\ref{5.3}), it must be that
    \begin{align*}
    \big(a_1+b_1\|u_n\|^{2(\theta_1-1)}\big)\langle u_n,u_n-u\rangle+\big(a_2+b_2\|v_n\|^{2(\theta_2-1)}\big)\langle v_n,v_n-v\rangle \rightarrow 0
    \end{align*}
    as $n\rightarrow\infty$. Since $a_1,\ a_2>0$, $b_1,\ b_2> 0$, and $W_n\rightharpoonup W$ in $H^1_0(\Omega,\mathbb{R}^2)$, it must be that
    \begin{align*}
    \langle u_n,u_n-u\rangle\rightarrow 0\ \ \mbox{and}\ \ \langle v_n,v_n-v\rangle \rightarrow 0,
    \end{align*}
    by the semicontinuity of norm, which means $\|u_n\|^2\rightarrow \|u\|^2$ and $\|v_n\|^2\rightarrow \|v\|^2$. Then by Br\'{e}zis-Lieb Lemma \cite{bl2}, we have $u_n\rightarrow u$ and $v_n\rightarrow v$ in $H^1_0(\Omega)$, i.e. $W_n\rightarrow W$ in $H^1_0(\Omega,\mathbb{R}^2)$. This proof is complete.
    \end{proof}

\noindent{\bfseries Proof of Theorem \ref{thmsesl}:} By (\ref{5.2}) and Lemma \ref{lem5.1}, using Minimax principle, there exists a critical point $W_0$ for $J_\lambda$ at level $c_{M,\lambda}$.
    We affirm that $W_0\neq (0,0)$. In fact, suppose by contradiction that $ W_0\equiv (0,0)$. Because $J_\lambda$ is continuous, $J_\lambda(W_n)\rightarrow c_{M,\lambda}$ and $W_n\rightarrow W_0$ in $H^1_0(\Omega,\mathbb{R}^2)$, we can know that $0<c_{M,\lambda}=\lim_{n\rightarrow \infty}J_\lambda(W_n)=J_\lambda(W_0)=J_\lambda((0,0))=0$, what is absurd. Thus, the proof of Theorem \ref{thmsesl} is complete.
    \qed

\subsection{\bfseries Critical case: Proof of Theorem \ref{thmcerd}}

    \begin{lemma}\label{ms1}
    Suppose that  $H_u,\ H_v$ have critical growth $(TM)^{c}$, and $(H_1)$, $(H_2)$ are satisfied. If $\{W_n\}\subset H^1_0(\Omega,\mathbb{R}^2)$ is a bounded Palais-Smale sequence for $J_\lambda$ at any finite level with
    \begin{align}\label{msVx}
    \limsup_{n\rightarrow\infty}\|W_n\|^2<\frac{4\pi}{\beta_0},
    \end{align}
    then $\{W_n\}$ possesses a strongly subsequence, i.e. the functional $J_\lambda$ satisfies the Palais-Smale condition.
    \end{lemma}

    \begin{proof}
    Let $\{W_n\}\subset H^1_0(\Omega,\mathbb{R}^2)$ is a bounded Palais-Smale sequence for $J_\lambda$ satisfying (\ref{msVx}), then up to a subsequence, there is $W=(u,v)\in H^1_0(\Omega,\mathbb{R}^2)$ such that
    \begin{align*}
    \begin{split}
    &W_n\rightharpoonup W\ \ \mbox{in}\ \ H^1_0(\Omega,\mathbb{R}^2).
    \end{split}
    \end{align*}
    By $J'_\lambda(W_n)\rightarrow 0$ in $\big(H^1_0(\Omega,\mathbb{R}^2)\big)^*$, we can get
    \begin{align*}
    \begin{split}
    o_n(1)=&
    \big(a_1+b_1\|u_n\|^{2(\theta_1-1)}\big)\langle u_n,u_n-u \rangle
    +\big(a_2+b_2\|v_n\|^{2(\theta_2-1)}\big)\langle v_n,v_n-v \rangle  \\
    & -\lambda\int_{\Omega}(W_n-W)\cdot\nabla H(x,W_n)dx,
    \end{split}
    \end{align*}
    that is,
    \begin{align}\label{ws3}
    \begin{split}
    &\big(a_1+b_1\|u_n\|^{2(\theta_1-1)}\big)(\|u_n\|^2-\|u\|^2)
    +\big(a_2+b_2\|v_n\|^{2(\theta_2-1)}\big)(\|v_n\|^2-\|v\|^2)  \\
    =&\lambda\int_{\Omega}(W_n-W)\cdot\nabla H(x,W_n)dx+o_n(1).
    \end{split}
    \end{align}
    From (\ref{msVx}), there exists $\zeta>0$ such that $\beta_0\|W_n\|^2<\zeta<4\pi$ for $n$ sufficiently large and also, there exist $\beta>\beta_0$ close to $\beta_0$ and $q>1$ close to $1$ such that $q\beta\|W_n\|^2\leq\zeta<4\pi$  for $n$ sufficiently large. Then by conditions $(TM)^{c}$ and $(H_2)$, we have
    \begin{align*}
    \begin{split}
    \left|\int_{\Omega}(W_n-W)\cdot\nabla H(x,W_n)dx\right|\leq C_1\int_{\Omega}|W_n-W|\exp(\beta|W_n|^2)dx,
    \end{split}
    \end{align*}
    and by the H\"{o}lder's inequality, Lemma \ref{lemtm4} and Sobolev embedding theorem, we can get that
    \begin{align*}
    \begin{split}
    \left|\int_{\Omega}(W_n-W)\cdot\nabla H(x,W_n)dx\right|
    \leq & C_1\|W_n-W\|_{\frac{q}{q-1}} \left(\int_{\Omega}|W_n-W|\exp(\beta|W_n|^2)dx\right)^{\frac{1}{q}} \\
    \leq & C_2\|W_n-W\|_{\frac{q}{q-1}}\to 0,
    \end{split}
    \end{align*}
    as $n\to \infty$. Then combining with (\ref{ws3}), we have $\|u_n\|\to \|u\|$ and $\|v_n\|\to \|v\|$. Then by Br\'{e}zis-Lieb Lemma \cite{bl2}, we have $u_n\rightarrow u,\ v_n\rightarrow v$ in $H^1_0(\Omega)$, i.e. $W_n\rightarrow W$ in $H^1_0(\Omega,\mathbb{R}^2)$  and the result follows.
    \end{proof}

    \begin{lemma}\label{lemcpsc}
    Suppose that  $H_u,\ H_v$ have critical growth $(TM)^{c}$, and $(H_1)$, $(H_2)$, $(H_3)$, $(SL)$ are satisfied. Moreover, if
    \begin{align}\label{mpllp}
    \begin{split}
    c_{M,\lambda}<\min\bigg\{\frac{1}{2}\Big(a_1\frac{2\pi}{\beta_0}+\frac{b_1}{\theta_1}\big(\frac{2\pi}{\beta_0}\big)^{\theta_1}\Big),
    \frac{1}{2}\Big(a_2\frac{2\pi}{\beta_0}+\frac{b_2}{\theta_2}\big(\frac{2\pi}{\beta_0}\big)^{\theta_2}\Big)\bigg\}.
    \end{split}
    \end{align}
    for some $\lambda>0$, then the functional $J_\lambda$ satisfies the Palais-Smale condition at level $c_{M,\lambda}$.
    \end{lemma}

    \begin{proof}
    From (\ref{5.2}) and Lemma \ref{lem4.1}, there exists a bounded Palais-Smale consequence $\{W_n\}$ for $J_\lambda$ at level $c_{M,\lambda}$. Then up to a subsequence, still denoted by $\{W_n\}$, for some $W_0\in H^1_0(\Omega,\mathbb{R}^2)$, one has
    \begin{align}\label{5.14}
    \begin{split}
    W_n & \rightharpoonup W_0\ \ \ \ \ \ \mbox{in}\ \ H^1_0(\Omega,\mathbb{R}^2), \\
    W_n & \rightarrow W_0\ \ \ \ \ \ \mbox{in}\ \ L^s(\Omega,\mathbb{R}^2),\ \ \mbox{for all}\ \ s\geq 1  \\
    W_n(x) & \rightarrow W_0(x)\ \ \mbox{in}\ \ \mathbb{R}^2, \ \ \mbox{for a.e.}\ \ x\in \Omega.
    \end{split}
    \end{align}
    Next, we will make some claims as follows:

\noindent{\bfseries Claim 1.} \ \ \ $W_0\neq (0,0)$.
    \begin{proof}
    Supposing by contradiction that $W_0\equiv (0,0)$. Hence, by Lemma \ref{lem4.4} and $(SL)$, $\int_{\Omega}H(x,W_n)\rightarrow 0$ and so
    \begin{align*}
    \begin{split}
    &\frac {1}{2}\Big(a_1\|u_n\|^2+\frac{b_1}{\theta_1}\|u_n\|^{2\theta_1}\Big)+\frac {1}{2}\Big(a_2\|v_n\|^2+\frac{b_2}{\theta_2}\|v_n\|^{2\theta_2}\Big)\rightarrow c_{M,\lambda} \\
    &<\min\bigg\{\frac{1}{2}\Big(a_1\frac{2\pi}{\beta_0}+\frac{b_1}{\theta_1}\big(\frac{2\pi}{\beta_0}\big)^{\theta_1}\Big),
    \frac{1}{2}\Big(a_2\frac{2\pi}{\beta_0}+\frac{b_2}{\theta_2}\big(\frac{2\pi}{\beta_0}\big)^{\theta_2}\Big)\bigg\}.
    \end{split}
    \end{align*}
     We can get that for $n$ large enough $\|u_n\|^2,\ \|v_n\|^2<\frac{2\pi}{\beta_0}$, so $\|W_n\|^2=\|u_n\|^2+\|v_n\|^2<\frac{4\pi}{\beta_0}$. %Therefore, there exists $n_0\in \mathbb{N}$ and $\alpha>0$ such that $\beta_0\|W_n\|^2<\alpha<4\pi$ for all $n>n_0$. Now, choosing $q>1$ close to 1 and $\beta>\beta_0$ close to $\beta_0$ so that we still have $q\beta\|W_n\|^2\leq\alpha<4\pi$. From this and by using $(H_2)$, H\"{o}lder's inequality, Lemma \ref{lemtm4} and Sobolev imbedding, we get
%    \begin{align*}
%    \begin{split}
%    \left|\int_{\Omega} W_n\cdot\nabla H(x,W_n) dx\right|&\leq C_1\int_{\Omega} |W_n|^2 dx+C_2\int_{\Omega} |W_n|e^{\beta|W_n|^2} dx \\
%    &\leq C_1\|W_n\|^2_2+C_3\|W_n\|_{\frac {q}{q-1}}\left(\int_{\Omega}e^{q\beta\|W_n\|^2\big|\frac {W_n}{\|W_n\|}\big|^2} dx\right)^{\frac{1}{q}} \\
%    &\leq C_1\|W_n\|^2_2+C_4\|W_n\|_{\frac {q}{q-1}}\rightarrow 0
%    \end{split}
%    \end{align*}
%    as $n\rightarrow \infty$. Since $\langle J_\lambda'(W_n),W_n\rangle\rightarrow 0$, then it follows that
%    \begin{align*}
%    \big(a_1+b_1\|u_n\|^{2(\theta_1-1)}\big)\|u_n\|^2+\big(a_2+b_2\|v_n\|^{2(\theta_2-1)}\big)\|v_n\|^2\rightarrow 0.
%    \end{align*}
%    Consequently, $\|u_n\|^2\to 0 $ and $\|v_n\|^2\rightarrow 0$,
Then by Lemma \ref{ms1}, we have $\|W_n\|^2\rightarrow 0$. Therefore $J_\lambda(W_n)\rightarrow 0$, what is absurd for $c_{M,\lambda}$ and we must have $W_0\neq (0,0)$.
    \end{proof}

\noindent{\bfseries Claim 2.} \ \ \ Let $\lim\inf_{n\rightarrow \infty}\|u_n\|^2=\sigma^2,\ \lim\inf_{n\rightarrow \infty}\|v_n\|^2=\tau^2$. Then $W_0=(u_0,v_0)$ is a weak solution of
    \begin{eqnarray}\label{wsy}
    \left\{ \arraycolsep=1.5pt
       \begin{array}{ll}
        -\big(a_1+b_1\sigma^{2(\theta_1-1)}\big)\Delta u=\lambda H_u(x,u,v)\ \ \ &\ \mbox{in}\ \ \ \Omega,\\[2mm]
        -\big(a_2+b_2\tau^{2(\theta_2-1)}\big)\Delta v=\lambda H_v(x,u,v)\ \ \ &\ \mbox{in}\ \ \ \Omega,\\[2mm]
        u=0, v=0\ \ \ \ &\ \mbox{on}\ \ \ \partial\Omega.
        \end{array}
    \right.
    \end{eqnarray}
    \begin{proof}
    By $J_\lambda'(W_n)\rightarrow 0$ and Remark \ref{rem4.3}, we get that
    \begin{align*}
    \begin{split}
    \big(a_1+b_1\sigma^{2(\theta_1-1)}\big)\int_{\Omega}\nabla u_0 \nabla \varphi +\big(a_2+b_2\tau^{2(\theta_2-1)}\big)\int_{\Omega}\nabla v_0 \nabla \psi -\lambda\int_{\Omega}\Phi\cdot\nabla H(x,W_0)=0,
    \end{split}
    \end{align*}
    $\mbox{for all}\ \Phi=(\varphi,\psi)\in C^\infty_0(\Omega,\mathbb{R}^2)$. Since $C^\infty_0(\Omega,\mathbb{R}^2)$ is dense in $H^1_0(\Omega,\mathbb{R}^2)$, we conclude this claim.
    \end{proof}

\noindent{\bfseries Claim 3.} \ \ \ $\lim_{n\rightarrow \infty}\|W_n\|^2-\|W_0\|^2<\frac{4\pi}{\beta_0}$.
    \begin{proof}
    Suppose by contradiction that $\lim_{n\rightarrow \infty}\|W_n\|^2-\|W_0\|^2=\sigma^2-\|u_0\|^2+\tau^2-\|v_0\|^2 \geq \frac{4\pi}{\beta_0}$ where $\sigma$ and $\tau$ are given in Claim 2. Thus, at least one between $\sigma^2-\|u_0\|^2$ and $\tau^2-\|v_0\|^2$ is not less then $\frac{2\pi}{\beta_0}$, let $\sigma^2-\|u_0\|^2 \geq \frac{2\pi}{\beta_0}$. Therefore, from (\ref{5.2}), we have
    \begin{align*}
    \begin{split}
    c_{M,\lambda}
    = & \lim_{n\rightarrow\infty}\big[J_\lambda(W_n)-\frac {1}{2\theta_{max}}\langle J_\lambda'(W_n),W_n\rangle\big] \\
    = & \lim_{n\rightarrow \infty} \Bigg[\frac {1}{2}\Big(a_1\|u_n\|^2+\frac{b_1}{\theta_1}\|u_n\|^{2\theta_1}\Big)
    +\frac {1}{2}\Big(a_2\|v_n\|^2+\frac{b_2}{\theta_2}\|v_n\|^{2\theta_2}\Big) \\
    & -\frac{1}{2\theta_{max}}\big(a_1+b_1\|u_n\|^{2(\theta_1-1)}\big)\|u_n\|^2
    -\frac{1}{2\theta_{max}}\big(a_2+b_2\|v_n\|^{2(\theta_2-1)}\big)\|v_n\|^2  \\
    & +\lambda\int_{\Omega}\left (\frac{1}{2\theta_{max}} W_n\cdot \nabla H(x,W_n) -H(x,W_n)\right )dx\Bigg] \\
    \geq & \lim\inf_{n\rightarrow \infty}\Big[\big(\frac {1}{2}-\frac{1}{2\theta_{max}}\big)a_1\|u_n\|^2
    +\big(\frac {1}{2\theta_1}-\frac{1}{2\theta_{max}}\big)b_1\|u_n\|^{2\theta_1}\Big] \\
    & +\lim\inf_{n\rightarrow \infty}\Big[\big(\frac {1}{2}-\frac{1}{2\theta_{max}}\big)a_2\|v_n\|^2
    +\big(\frac {1}{2\theta_1}-\frac{1}{2\theta_{max}}\big)b_2\|v_n\|^{2\theta_2}\Big] \\
    & +\lambda\lim\inf_{n\rightarrow \infty}\int_{\Omega}\left (\frac{1}{2\theta_{max}} W_n\cdot \nabla H(x,W_n) -H(x,W_n)\right)dx \\
    \geq & \lim\inf_{n\rightarrow \infty}\Big[\big(\frac {1}{2}-\frac{1}{2\theta_{max}}\big)a_1\|u_n\|^2
    +\big(\frac {1}{2\theta_1}-\frac{1}{2\theta_{max}}\big)b_1\|u_n\|^{2\theta_1}\Big] \\
    & +\frac{\lambda}{2\theta_{max}} \lim\inf_{n\rightarrow \infty}\int_{\Omega}\big(W_n\cdot \nabla H(x,W_n) -2\theta_{max} H(x,W_n)\big)dx. \\
    \end{split}
    \end{align*}
    Then by $(H_3)$, Remark \ref{rem4.3} and Lemma \ref{lem4.4}, we obtain
    \begin{align*}
    \begin{split}
    c_{M,\lambda}
    \geq &  \Big[\big(\frac {1}{2}-\frac{1}{2\theta_{max}}\big)a_1\frac{2\pi}{\beta_0}
    +\big(\frac {1}{2\theta_1}-\frac{1}{2\theta_{max}}\big)b_1(\frac{2\pi}{\beta_0})^{\theta_1}\Big] \\
    & +\frac{\lambda}{2\theta_{max}} \lim\inf_{n\rightarrow \infty}\int_{\Omega}\big(W_n\cdot \nabla H(x,W_n) -W_0\cdot \nabla H(x,W_0)\big)dx. \\
    = & \frac {1}{2}\Big[a_1\frac{2\pi}{\beta_0}+\frac{b_1}{\theta_1}(\frac{2\pi}{\beta_0})^{\theta_1}\Big] \\
    & -\frac{1}{2\theta_{max}} \Big[a_1\frac{2\pi}{\beta_0}+b_1(\frac{2\pi}{\beta_0})^{\theta_1}
    -\lambda\lim\inf_{n\rightarrow \infty}\int_{\Omega}\big(W_n\cdot \nabla H(x,W_n) -W_0\cdot \nabla H(x,W_0)\big)dx\Big].
    \end{split}
    \end{align*}
    Here, we assert
    \begin{align*}
    0\geq a_1\frac{2\pi}{\beta_0}+b_1(\frac{2\pi}{\beta_0})^{\theta_1}-\lambda \lim\inf_{n\rightarrow \infty}\int_{\Omega}\big(W_n\cdot \nabla H(x,W_n) -W_0\cdot \nabla H(x,W_0)\big)dx.
    \end{align*}
    Indeed, $\langle J_\lambda'(W_n),W_n\rangle\rightarrow 0$ and Claim 2 indicate that
    \begin{align*}
    \begin{split}
    0&=\Big(a_1+b_1\sigma^{2(\theta_1-1)}\Big)\sigma^2+\Big(a_2+b_2\tau^{2(\theta_2-1)}\Big)\tau^2
    -\lambda\lim\inf_{n\rightarrow \infty}\int_{\Omega}W_n\cdot \nabla H(x,W_n)dx. \\
    0&=\Big(a_1+b_1\sigma^{2(\theta_1-1)}\Big)\|u_0\|^2+\Big(a_2+b_2\tau^{2(\theta_1-1)}\Big)\|v_0\|^2
    -\lambda\int_{\Omega}W_0\cdot \nabla H(x,W_0)dx.
    \end{split}
    \end{align*}
    Subtracting the second equality from the first one, we get
    \begin{align*}
    \begin{split}
    0= & \Big(a_1+b_1\sigma^{2(\theta_1-1)}\Big)(\sigma^2-\|u_0\|^2)
    +\Big(a_2+b_2\tau^{2(\theta_2-1)}\Big)(\tau^2-\|v_0\|^2) \\
    & -\lambda\lim\inf_{n\rightarrow \infty}\int_{\Omega}\big(W_n\cdot \nabla H(x,W_n) -W_0\cdot \nabla H(x,W_0)\big)dx \\
    \geq & \Big(a_1+b_1\sigma^{2(\theta_1-1)}\Big)\frac{2\pi}{\beta_0}
    -\lambda\lim\inf_{n\rightarrow\infty}\int_{\Omega}\big(W_n\cdot \nabla H(x,W_n) -W_0\cdot \nabla H(x,W_0)\big)dx\\
    \geq & a_1\frac{2\pi}{\beta_0}+b_1(\frac{2\pi}{\beta_0})^{\theta_1}
    -\lambda\lim\inf_{n\rightarrow \infty}\int_{\Omega}\big(W_n\cdot \nabla H(x,W_n) -W_0\cdot \nabla H(x,W_0)\big)dx.
    \end{split}
    \end{align*}
    This shows the assertion. Then we can conclude that
    \begin{align*}
    c_{M,\lambda}\geq \frac {1}{2}\Big[a_1\frac{2\pi}{\beta_0}
    +\frac{b_1}{\theta_1}(\frac{2\pi}{\beta_0})^{\theta_1}\Big],
    \end{align*}
    which contradicts the assumption. Thus this claim is proved.
    \end{proof}

\noindent{\bfseries Claim 4.} \ \ \ $J_\lambda(W_0)=c_{M,\lambda}$.
    \begin{proof}
    By using Lemma \ref{lem4.4}, and semicontinuity of norm, we have that $J_\lambda(W_0)\leq c_{M,\lambda}$. We will show that the case $J_\lambda(W_0)<c_{M,\lambda}$ can not occur. We suppose $J_\lambda(W_0)<c_{M,\lambda}$, and define $\sigma$ and $\tau $ as in Claim 2, then we get $\sigma>\|u_0\|$ or $\tau>\|v_0\|$. Next, defining $Z_n=\frac{W_n}{\|W_n\|}$ and $Z_0=\frac{W_0}{\sqrt{\sigma^2+\tau^2}}=\frac{W_0}{\lim_{n\rightarrow \infty}\|W_n\|}$, we have $Z_n\rightharpoonup Z_0$ in $H^1_0(\Omega,\mathbb{R}^2)$ and $\|Z_0\|<1$, then Lemma \ref{lemtm5} which shows that
    \begin{align}\label{5.16}
    \sup_n \int_{\Omega}\exp(p|Z_n|^2)dx<\infty,\ \ \forall p<\frac{4\pi} {1-\|Z_0\|^2}.
    \end{align}
    Since $\sigma^2+\tau^2=\frac{\sigma^2+\tau^2-\|W_0\|^2}{1-\|Z_0\|^2}$, by Claim 3, it follows that
    \begin{align*}
    \lim_{n\rightarrow \infty}\|W_n\|^2 <\frac{4\pi/\beta_0}{1-\|Z_0\|^2}.
    \end{align*}
    Thus, there exists $\alpha>0$ such that $\beta_0\|W_n\|^2<\alpha<\frac{4\pi}{1-\|Z_0\|^2}$ for $n$ sufficiently large. For $q>1$ close to 1 and $\beta>\beta_0$ close to $\beta_0$ we still have $q\beta\|W_n\|^2\leq\alpha<\frac{4\pi}{1-\|Z_0\|^2}$, and provoking (\ref{5.16}), for some $C>0$ and $n$ large enough, we conclude that
    \begin{align*}
    \int_{\Omega}\exp(q\beta|W_n|^2)dx\leq \int_{\Omega}\exp(\alpha|Z_n|^2)dx\leq C.
    \end{align*}
    Then, by using $(H_2)$ and critical growth condition $(TM)^c$, (\ref{5.14}), %Lemma \ref{lemtm2},
    H\"{o}lder's inequality and Sobolev imbedding, we can obtain
    \begin{align*}
    \begin{split}
    \left|\int_{\Omega} (W_n-W_0)\cdot\nabla H(x,W_n) dx\right|
    \leq & C_1\int_{\Omega} |W_n|\cdot|W_n-W_0| dx \\
    & +C_2\int_{\Omega}|W_n-W_0|\exp(\beta|W_n|^2) dx \\
    \leq & C_1\|W_n-W_0\|_2\|W_n\|_2+C_3\|W_n-W_0\|_{\frac{q}{q-1}}\rightarrow 0.
    \end{split}
    \end{align*}
    Since $\langle J_\lambda'(W_n),W_n-W_0\rangle\rightarrow 0$, it follows that
    \begin{align*}
    \big(a_1+b_1\|u_n\|^{2(\theta_1-1)}\big)\langle u_n,u_n-u_0 \rangle+\big(a_2+b_2\|v_n\|^{2(\theta_2-1)}\big)\langle v_n,v_n-v_0 \rangle \rightarrow 0.
    \end{align*}
    On the other hand,
    \begin{align*}
    \begin{split}
    \big(a_1+b_1\|u_n\|^{2(\theta_1-1)}\big)\langle u_n,u_n-u_0 \rangle&=\big(a_1+b_1\|u_n\|^{2(\theta_1-1)}\big)\|u_n\|^2-\big(a_1+b_1\|u_n\|^{2(\theta_1-1)}\big)\langle u_n,u_0 \rangle \\
    &\rightarrow\big(a_1+b_1\sigma^{2(\theta_1-1)}\big)\sigma^2-\big(a_1+b_1\sigma^{2(\theta_1-1)}\big)\|u_0\|^2, \\
    \big(a_2+b_2\|v_n\|^{2(\theta_2-1)}\big)\langle v_n,v_n-v_0 \rangle&=\big(a_2+b_2\|v_n\|^{2(\theta_2-1)}\big)\|v_n\|^2-\big(a_2+b_2\|v_n\|^{2(\theta_2-1)}\big)\langle v_n,v_0 \rangle \\
    &\rightarrow\big(a_2+b_2\tau^{2(\theta_2-1)}\big)\tau^2-\big(a_2+b_2\tau^{2(\theta_2-1)}\big)\|v_0\|^2,
    \end{split}
    \end{align*}
    which implies that $\sigma=\|u_0\|$ and $\tau=\|v_0\|$, i.e. $\|W_n\|\rightarrow \|W_0\|$. Then, we must have $J_\lambda(W_0)=c_{M,\lambda}$, what is an absurd. Thus, this claim is proved.
    \end{proof}

\noindent{\bfseries Finalizing the proof of Lemma \ref{lemcpsc}}: By Lemma \ref{lem4.4} and Claim 4, we can know that \begin{align*}
    \begin{split}
    \Big(a_1\sigma^2+\frac{b_1}{\theta_1}\sigma^{2\theta_1}\Big)+&\Big(a_2\tau^2+\frac{b_2}{\theta_2}\tau^{2\theta_2}\Big)
    =\Big(a_1\|u_0\|^2+\frac{b_1}{\theta_1}\|u_0\|^{2\theta_1}\Big)+\Big(a_2\|v_0\|^2+\frac{b_2}{\theta_2}\|v_0\|^{2\theta_2}\Big).
    \end{split}
    \end{align*}
    By the semicontinuity of norm, we have $\sigma=\|u_0\|,\ \tau=\|v_0\|$. Then by Br\'{e}zis-Lieb Lemma \cite{bl2}, we have $u_n\rightarrow u_0,\ v_n\rightarrow v_0$ in $H^1_0(\Omega)$, i.e. $W_n\rightarrow W_0$ in $H^1_0(\Omega,\mathbb{R}^2)$.
    \end{proof}

    \begin{lemma}\label{lemc*y}
    Suppose that $H_u,\ H_v$ have the critical growth $(TM)^{c}$ and $(H_1)$, $(H_2)$, $(H_3)$, $(SL)$ are satisfied. Then, we have $0<c_{M,\lambda}\rightarrow 0$ as $\lambda\rightarrow +\infty$.
    \end{lemma}

    \begin{proof}
    Previous lemmas and the definition of $c_{M,\lambda}$ (see (\ref{5.1})) imply $c_{M,\lambda}>0$ for each $\lambda>0$. Fixing any nonnegative functions $\Phi=(\phi,\varphi)\in C^\infty_0(\Omega,\mathbb{R}^2)\backslash \{(0,0)\}$, then there exists $t_\lambda>0$ such that $J_\lambda(t_\lambda\Phi)=\max_{t>0}J_\lambda(t\Phi)$. Since $\frac{dJ_\lambda(t\Phi)}{dt}\big|_{t=t_\lambda}=0$, then by $(H_3)$ and (\ref{rf3arb}) we have
    \begin{align}\label{c*y1}
    \begin{split}
    a_1t_\lambda^2\|\phi\|^2+a_2t_\lambda^2\|\varphi\|^2+b_1 t_\lambda^{2\theta_1}\|\phi\|^{2\theta_1}
    +b_2 t_\lambda^{2\theta_2}\|\varphi\|^{2\theta_2}
    = & \lambda\int_{\Omega}t_\lambda\Phi \cdot \nabla H(x,t_\lambda \Phi)dx \\
    %\geq & 2\theta_{max}\lambda\int_{\Omega} H(x,t_\lambda \Phi)dx \\
    \geq & \lambda \left(C_1 t_\lambda^{\xi} \int_{\Omega}|\Phi|^{\xi} dx-C_2\right)
    \end{split}
    \end{align}
    for some $C_1,\ C_2>0$ with $\xi>2\theta_{max}$. This implies that $t_\lambda$ is bounded from above uniformly in $\lambda$. In addition, we claim $t_\lambda\rightarrow 0 $ as $\lambda \rightarrow +\infty$. If not, we can choose a sequence $\{\lambda_n\}\subset \mathbb{R}^+$ and a constant $t_0>0$ such that $\lambda_n \rightarrow +\infty$ and $t_{\lambda_n}\rightarrow t_0>0$. Since $\Phi\in C^\infty_0(\Omega,\mathbb{R}^2)$, we have $\nabla H(x,t_{\lambda_n} \Phi)\rightarrow \nabla H(x,t_0 \Phi)$ in $L^1(\Omega,\mathbb{R}^2)$ by \cite[Lemma 2.1]{dmr1}. Then we get
    \begin{align*}
    \begin{split}
    \int_{\Omega}\Phi \cdot \nabla H(x,t_{\lambda_n} \Phi)dx \rightarrow \int_{\Omega}\Phi \cdot \nabla H(x,t_0 \Phi)dx.
    \end{split}
    \end{align*}
    Noting that the right side is strictly positive by $(H_3)$. Then it follows from (\ref{c*y1}) that
    \begin{align*}
    \begin{split}
    a_1t_{\lambda_n}^2\|\phi\|^2+a_2t_{\lambda_n}^2\|\varphi\|^2+b_1 t_{\lambda_n}^{2\theta_1}\|\phi\|^{2\theta_1}
    +b_2 t_{\lambda_n}^{2\theta_2}\|\varphi\|^{2\theta_2} \rightarrow +\infty
    \end{split}
    \end{align*}
    as $n\rightarrow \infty$, this is impossible. Therefore by $(H_1)$, we can conclude that
    \begin{align*}
    \begin{split}
    c_{M,\lambda}\leq J_\lambda(t_\lambda\Phi) \leq
    t^2_\lambda\frac{a_1\|\phi\|^2}{2}+t^2_\lambda\frac{a_2\|\varphi\|^2}{2}
    +t_{\lambda}^{2\theta_1}\frac{b_1 \|\phi\|^{2\theta_1}}{2\theta_1}
    +t_{\lambda}^{2\theta_2}\frac{b_2 \|\varphi\|^{2\theta_2}}{2\theta_2} \rightarrow 0
    \end{split}
    \end{align*}
    as $\lambda\rightarrow +\infty$. Thus, this proof is complete.
    \end{proof}

\noindent{\bfseries Proof of Theorem \ref{thmcerd}:} By Lemma \ref{lemc*y}: $0<c_{M,\lambda}\rightarrow 0$ as $\lambda\rightarrow +\infty$, thus, there exists a constant $\lambda_*>0$ large enough such that for all $\lambda\geq \lambda_*$, it holds that
    \begin{align*}
    c_{M,\lambda}<\min\bigg\{\frac{1}{2}\Big(a_1\frac{2\pi}{\beta_0}+\frac{b_1}{\theta_1}\big(\frac{2\pi}{\beta_0}\big)^{\theta_1}\Big),
    \frac{1}{2}\Big(a_2\frac{2\pi}{\beta_0}+\frac{b_2}{\theta_2}\big(\frac{2\pi}{\beta_0}\big)^{\theta_2}\Big)\bigg\}.
    \end{align*}
    Lemma \ref{lemcpsc} indicts that the functional $J_\lambda$ satisfies the Palais-Smale condition at level $c_{M,\lambda}$
    for each $\lambda\geq \lambda_*$.

    Then by using Minimax principle, there exists a critical point $W_*$ for $J_\lambda$ at level $c_{M,\lambda}$ for each $\lambda\geq \lambda_*$.
    We affirm that $W_*\neq (0,0)$. In fact, suppose by contradiction that $ W_*\equiv (0,0)$. Because $J_\lambda$ is continuous, $J_\lambda(W_n)\rightarrow c_{M,\lambda}$ and $W_n\rightarrow W_*$ in $H^1_0(\Omega,\mathbb{R}^2)$, we can know that $0<c_{M,\lambda}=\lim_{n\rightarrow \infty}J_\lambda(W_n)=J_\lambda(W_*)=J_\lambda((0,0))=0$, what is absurd. Thus, the proof of Theorem \ref{thmcerd} is complete.
    \qed

\section{{\bfseries Multiplicity results}}\label{vfh6}

    The next lemma motivates us to investigate the multiplicity of solutions in the asymptotically case.
   \begin{lemma}\label{lemgc3}
    Suppose that $(H_1)$, $(H_2)$, $(AL)$ hold and $H_u,\ H_v$ have subcritical growth $(TM)^{sc}$ (or critical growth $(TM)^{c}$). Then, there exists a constant $b_0>0$ such that for every $b_{min}> b_0$, we have a constant $\Lambda_0(b_{min})>\frac{a_{max}}{a_{min}}\lambda_1$ such that there exist constants $\tau_1>0$ and $0<\rho_1<\rho_2<1$ such that
    \begin{align*}
    J_\lambda(W)\geq \tau_1,\ \ \forall\ W\in H^1_0(\Omega,\mathbb{R}^2)\ \mbox{with}\ \rho_1\leq \|W\|\leq \rho_2,
    \end{align*}
    and
    \begin{align*}
    -\infty<\inf_{\|W\|\leq \rho_1} J_\lambda(W)<0,
    \end{align*}
    for all $\frac{a_{max}}{a_{min}}\lambda_1<\lambda<\Lambda_0(b_{min})$, where $b_{min}>b_0$. Furthermore, we can take the constant $\Lambda_0(b_{min})$ depending only on $b_{min}$ such that $\Lambda_0(b_{min})\rightarrow \infty$ as $b_{min}\rightarrow \infty$.
    \end{lemma}

    \begin{proof}
    By assumptions $(AL)$, $(TM)^{sc}$ (or $(TM)^{c}$) and $(H_2)$, we can choose $\beta=\max\{1,2\beta_0\}$ (that is, for subcritical case we take $\beta=1$ and for critical case we take $\beta=2\beta_0$) such that
    \begin{align*}
    H(x,W) \leq \frac{(a_{min}+1)|W|^2}{2}+C|W|^2 e^{\beta |W|^2},\ \forall (x,W)\in \Omega\times \mathbb{R}^2.
    \end{align*}
    Thus by Lemmas \ref{lemtm5}, \ref{lemiit} and H\"{o}lder's inequality, we obtain
    \begin{align}\label{gc3.3}
    \begin{split}
    J_\lambda(W)
    = & \frac {1}{2}\Big(a_1\|u\|^2+\frac{b_1}{\theta_1}\|u\|^{2\theta_1}\Big)+\frac {1}{2}\Big(a_2\|v\|^2+\frac{b_2}{\theta_2}\|v\|^{2\theta_2}\Big)-\lambda\int_{\Omega}H(x,W)dx \\
    %\geq & \frac {1}{2}\Big(\frac{b_1}{\theta_1}\|u\|^{2\theta_1}+\frac{b_2}{\theta_2}\|v\|^{2\theta_2}\Big)+\frac {1}{2}\Big(a_1\|u\|^2+a_2\|v\|^2\Big) \\
    %& -\frac{\lambda (a_{min}+1)}{2}\int_{\Omega}|W|^2dx-\lambda C\int_{\Omega}|W|^2 \exp(\beta |W|^2)dx \\
    \geq & \frac{b_{min}}{2\theta_{max}}\big(\|u\|^{2\theta_{max}}+\|v\|^{2\theta_{max}}\big)+\frac{a_{min}}{2}\big(\|u\|^{2}+\|v\|^{2}\big) \\
    & -\frac{\lambda (a_{min}+1)}{2}\int_{\Omega}|W|^2dx-\lambda C\int_{\Omega}|W|^2\exp(\beta |W|^2)dx \\
    \geq & \frac{b_{min}}{2\theta_{max}(2^{\theta_{max}}-1)}\|W\|^{2\theta_{max}}-\left[\frac{\lambda (a_{min}+1)}{2\lambda_1}-\frac{a_{min}}{2}\right]\|W\|^2 \\
    & -\lambda C\Big(\int_{\Omega}|W|^{4}dx\Big)^{1/2}\left(\int_{\Omega}\exp\left(2\beta \|W\|^2(W/\|W\|)^2\right)dx\right)^{1/2},
    \end{split}
    \end{align}
    for all $W\in H^1_0(\Omega,\mathbb{R}^2)$ with $\|W\|\leq 1$.
    If $\|W\|^2\leq \min\{1,\frac{\pi}{\beta_0}\}$, then from Lemma \ref{lemtm4}, it holds that
    \begin{align}\label{gc3.3dd}
    \begin{split}
    J_\lambda(W)\geq &\frac{b_{min}}{2\theta_{max}(2^{\theta_{max}}-1)}\|W\|^{2\theta_{max}}
    -\Big[\frac{\lambda (a_{min}+1)}{2\lambda_1}-\frac{a_{min}}{2}\Big]\|W\|^2-\lambda C'\|W\|^2 \\
     \geq &\frac{b_{min}}{2\theta_{max}(2^{\theta_{max}}-1)}\|W\|^{2\theta_{max}}-\lambda C_*\|W\|^2 \\
    = &\frac{b_{min}}{2\theta_{max}(2^{\theta_{max}}-1)}\|W\|^2
    \left[\|W\|^{2\theta_{max}-2}-\frac{2\theta_{max}(2^{\theta_{max}}-1)\lambda C_*}{b_{min}}\right],
    \end{split}
    \end{align}
    for some constant $C_*>0$. Then we choose
    \begin{align}\label{deflbm}
    \Lambda_0(b_{min}):=\frac{b_{min}}{2\theta_{max}(2^{\theta_{max}}-1)C_*} \left(\min\{1,\frac{\pi}{\beta_0}\}\right)^{\theta_{max}-1}.
    \end{align}
    It follows that there exists a constant
    \begin{equation}\label{defb0}
    b_0:=2\theta_{max}(2^{\theta_{max}}-1)C_*\left(\min\{1,\frac{\pi}{\beta_0}\}\right)^{1-\theta_{max}}\frac{a_{max}}{a_{min}}\lambda_1>0
    \end{equation}
    such that
    \begin{equation*}
     \Lambda_0(b_{min})>\frac{a_{max}}{a_{min}}\lambda_1 \quad\mbox{for all}\quad b_{min}> b_0.
    \end{equation*}
    Taking $b_{min}> b_0$ and $\frac{a_{max}}{a_{min}}\lambda_1\leq \lambda<\Lambda_0(b_{min})$ we have
    \[\frac{2\theta_{max}(2^{\theta_{max}}-1)C_*\lambda}{b_{min}}<\left(\min\{1,\frac{\pi}{\beta_0}\}\right)^{\theta_{max}-1}.\]
     Then from (\ref{gc3.3dd}), we can choose $0<\rho_1<\rho_2<1$ satisfying
    \[\left[\frac{2\theta_{max}(2^{\theta_{max}}-1)C_*\lambda}{b_{min}}\right]^{\frac{1}{2(\theta_{max}-1)}}<\rho_1<\rho_2<\left(\min\{1,\frac{\pi}{\beta_0}\}\right)^{\frac{1}{2}}\]
     such that
    \begin{align}\label{gc3.4}
    J_\lambda(W)\geq \tau_1,\ \ \mbox{for all}\ \ W\in H^1_0(\Omega,\mathbb{R}^2)\ \mbox{with}\ \rho_1\leq \|W\|\leq \rho_2,
    \end{align}
    for some constant $\tau_1>0$.
    On the other hand, we claim
    \begin{align}\label{gc3.5}
    J_\lambda(t\Phi_1)<0,
    \end{align}
    if $\frac{a_{max}}{a_{min}}\lambda_1<\lambda$ and $t>0$ is sufficiently small, where $\Phi_1=(\phi_1,\varphi_1)$ is the eigenfunction of $\lambda_1$, i.e. $\lambda_1=\frac {\|\Phi_1\|^2}{\int_{\Omega}|\Phi_1|^2dx}$. Actually, fix $\epsilon>0$ such that $\lambda(a_{min}-\epsilon)>\lambda_1a_{max}$.\ $(AL)$ implies that there exists a constant $S_\epsilon$ such that
    \begin{align*}
    H(x,W)\geq \frac{a_{min}-\epsilon}{2}|W|^2,\ \mbox{if}\ 0\leq |W|\leq S_\epsilon \ \mbox{ uniformly for}\ x\in\Omega.
    \end{align*}
    Set $\kappa_\epsilon>0$ such that $\kappa_\epsilon \sup_{x\in \Omega}|\Phi_1|\leq S_\epsilon$. Consequently, by $(H_1)$ we can get that
    \begin{align*}
    \begin{split}
    J_\lambda(t\Phi_1)\leq & \frac{t^2 a_{max}}{2}\|\Phi_1\|^2-
    \frac{t^2 \lambda(a_{min}-\epsilon)}{2}\int_{\Omega}|\Phi_1|^2dx+t^{2\theta_1}\frac{b_1 \|\phi_1\|^{2\theta_1}}{2\theta_1}
    +t^{2\theta_2}\frac{b_2 \|\varphi_1\|^{2\theta_2}}{2\theta_2} \\
    = &-\frac{t^2}{2}\left[\frac{\lambda(a_{min}-\epsilon)}{\lambda_1a_{max}}-1\right]a_{max}\|\Phi_1\|^2
    +t^{2\theta_1}\frac{b_1 \|\phi_1\|^{2\theta_1}}{2\theta_1}+t^{2\theta_2}\frac{b_2 \|\varphi_1\|^{2\theta_2}}{2\theta_2},
    \end{split}
    \end{align*}
    for all $0<t<\kappa_\epsilon$. Since $2\theta_1,2\theta_2> 2$, this concludes the claim. Thus, (\ref{gc3.4}) and (\ref{gc3.5}) end the proof.
    \end{proof}

%\section{{\bfseries Analysis of compactness with parameter $\lambda$}}\label{vfh7}

    %In this section, we give the compactness result for $(PS)$ sequences for $J_\lambda$. And those proofs are similar to the proof of Theorem , we only give the key points.

 In the rest of this section, we always suppose that $\frac{a_{max}}{a_{min}}\lambda_1<\lambda<\Lambda_0(b_{min})$, where $\Lambda_0(b_{min})$ is obtained in Lemma \ref{lemgc3} and $b_{min}> b_0$.

    In order to obtain a weak solution with negative energy, by Lemma \ref{lemgc3}, we define that
    \begin{equation*}
    -\infty<c_{0,\lambda}:=\inf_{\|W\|\leq \rho_1}J_\lambda(W)<0,
    \end{equation*}
    where $\rho_1$ is given as in Lemma \ref{lemgc3}. Since $\overline{B}_{\rho_1}$ is a complete metric space with the metric given by the norm of $H^1_0(\Omega,\mathbb{R}^2)$, convex and the functional $J_\lambda$ is of class $C^1$ and bounded below on $\overline{B}_{\rho_1}$, by the Ekeland variational principle, there exists a sequence $\{W^0_n\}$ in $\overline{B}_{\rho_1}$ such that
    \begin{equation}\label{7.1}
    J_\lambda(W^0_n)\rightarrow c_{0,\lambda}=\inf_{\|W\|\leq \rho_1}J_\lambda(W)<0\ \ \mbox{and}\ \ J_\lambda'(W^0_n)\rightarrow 0.
    \end{equation}
    And in order to obtain another weak solution with positive energy, according to Lemmas \ref{lemgc2} and \ref{lemgc3}, let
    \begin{equation*}
    c_{*,\lambda}=\inf_{\gamma\in\Gamma_\lambda}\max_{t\in [0,1]}J_\lambda(\gamma(t))>0
    \end{equation*}
    be the minimax level of $J_\lambda$, where $\Gamma_\lambda=\{\gamma\in C\big([0,1],H^1_0(\Omega,\mathbb{R}^2)\big):\gamma(0)=0,\gamma(1)=E_0\}$, where $E_0$ is given in Lemma \ref{lemgc2} and here we choose $\|E_0\|>\max\{\rho_0,\rho_2\}$. Therefore, by using the Mountain-pass theorem, there exists a sequence $\{W^*_n\} \subset H^1_0(\Omega,\mathbb{R}^2)$ satisfying
    \begin{equation}\label{7.2}
    J_\lambda(W^*_n)\rightarrow c_{*,\lambda}\ \ \mbox{and}\ \ J_\lambda'(W^*_n)\rightarrow 0.
    \end{equation}

\noindent{\bfseries Proof of Theorem \ref{thmscerm}:} Similarly to the proof of Lemma \ref{lem5.1}, by using the subcritical growth condition $(TM)^{sc}$, we can get that the functional $J_\lambda$ satisfies the Palais-Smale condition. Then, similarly to the proof of Theorem \ref{thmsesl}, by (\ref{7.1}), (\ref{7.2}), using Minimax principle, there exist two nontrivial critical points $W_{0,\lambda},\ W_{*,\lambda}$ for $J_\lambda$ at the levels of $c_{0,\lambda}$ and $c_{*,\lambda}$, respectively. By compactness condition and the definitions of $c_{0,\lambda}$ and $c_{*,\lambda}$, it is easy to verify that $W_{0,\lambda},\ W_{*,\lambda}$ are different. Thus, the proof of Theorem \ref{thmscerm} is complete.
    \qed

\noindent{\bfseries Proof of Theorem \ref{thmcerm}:} From (\ref{7.1}) and Lemma \ref{lemgc3}, there exists a sequence $\{W^0_n\}$ in $\overline{B}_{\rho_1}$ such that
    \begin{align*}
    J_\lambda(W^0_n)\rightarrow c_{0,\lambda}=\inf_{\|W\|\leq \rho_1}J_\lambda(W)<0\ \ \mbox{and}\ \ J_\lambda'(W^0_n)\rightarrow 0,
    \end{align*}
    where $\rho^2_1\leq \min\{1,\frac{\pi}{\beta_0}\}$ is given in Lemma \ref{lemgc3}. Thus $\|W^0_n\|^2 \leq \min\{1,\frac{\pi}{\beta_0}\}<\frac{4\pi}{\beta_0}$ for all $n\in\mathbb{N}$. Then by Lemma \ref{ms1}, the functional $J_\lambda$ satisfies the Palais-Smale condition at level $c_{0,\lambda}$. Thus, we can get one nontrivial critical point $W_{0,\lambda}$ for $J_\lambda$ at level $c_{0,\lambda}$ and $J_\lambda(W_{0,\lambda})=c_{0,\lambda}<0$. Therefore, the proof of Theorem \ref{thmcerm} is complete.
    \qed

\bigskip

\noindent{\bfseries Acknowledgements:}

The authors have been supported by National Nature Science Foundation of China 11971392, Natural Science Foundation of Chongqing, China cstc2019jcyjjqX0022  and Fundamental Research
Funds for the Central Universities XDJK2019TY001.

    \end{document}